\documentclass[a4paper,DIV=11,abstracton]{scrartcl}

\usepackage{color}
\usepackage{float}
\usepackage{mathtools}
\usepackage{amsmath}
\usepackage{amsthm}
\usepackage{amssymb}
\usepackage{graphicx}
\usepackage{esint}

\usepackage
[
    unicode=true,
    bookmarks=false,
    breaklinks=false,
    pdfborder={0 0 1},
    backref=false,colorlinks=true
]
{hyperref}

\hypersetup{
    pdffitwindow=false,
    pdfpagelayout=SinglePage,
    linkcolor=blue,
    urlcolor=blue!30!black,
    citecolor=green!50!black
}

\makeatletter

\pdfpageheight\paperheight
\pdfpagewidth\paperwidth

\floatstyle{ruled}
\newfloat{algorithm}{tbp}{loa}
\providecommand{\algorithmname}{Algorithm}
\floatname{algorithm}{\protect\algorithmname}

\theoremstyle{plain}
\newtheorem{thm}{\protect\theoremname}[section]
\theoremstyle{definition}
\newtheorem{example}[thm]{\protect\examplename}
\theoremstyle{definition}
\newtheorem{defn}[thm]{\protect\definitionname}
\theoremstyle{definition}
\newtheorem{rem}[thm]{\protect\remarkname}
\theoremstyle{plain}
\newtheorem{prop}[thm]{\protect\propositionname}

\usepackage{graphicx}
\usepackage{dsfont}
\usepackage{xcolor}
\usepackage{algorithm}
\usepackage{algpseudocode}
\usepackage{authblk}
\usepackage{enumitem}
\usepackage[bf]{caption}
\usepackage[T1]{fontenc}

\DeclareMathAlphabet{\mathpzc}{OT1}{pzc}{m}{it}

\setcapindent{0pt}

\setlist[enumerate]{leftmargin=0pt,itemindent=1.5em}

\allowdisplaybreaks

\newcommand{\subfiguretitle}[1]{{\scriptsize{#1}} \\[1mm]}
\newcommand{\tsub}[1]{\text{\tiny{#1}}}                         
\providecommand{\norm}[1]{\left\lVert #1 \right\rVert}          

\newcommand{\xqed}[1]{\leavevmode\unskip\penalty9999 \hbox{}\nobreak\hfill \quad\hbox{#1}}
\newcommand{\exampleSymbol}{\xqed{$\triangle$}}

\renewcommand*{\env@matrix}[1][*\c@MaxMatrixCols c]{%
  \hskip -\arraycolsep
  \let\@ifnextchar\new@ifnextchar
  \array{#1}}

\mathtoolsset{centercolon} 

\author[1]{Stefan Klus}
\author[1]{Feliks N\"uske}
\author[1]{P\'eter Koltai}
\author[1]{Hao Wu}
\author[2,3]{Ioannis Kevrekidis}
\author[1,3]{Christof Sch\"utte}
\author[1]{Frank No\'e}
\affil[1]{\normalsize Department of Mathematics and Computer Science, Freie Universit\"at Berlin, Germany}
\affil[2]{\normalsize Department of Chemical and Biological Engineering \& Program in Applied and Computational Mathematics, Princeton University, USA}
\affil[3]{\normalsize Zuse Institute Berlin, Germany}
\date{}

\makeatother

\providecommand{\definitionname}{Definition}
\providecommand{\examplename}{Example}
\providecommand{\propositionname}{Proposition}
\providecommand{\remarkname}{Remark}
\providecommand{\theoremname}{Theorem}

\begin{document}

\title{Data-driven model reduction and\\transfer operator approximation}
\maketitle

\begin{abstract}
In this review paper, we will present different data-driven dimension reduction techniques for dynamical systems that are based on transfer operator theory as well as methods to approximate transfer operators and their eigenvalues, eigenfunctions, and eigenmodes. The goal is to point out similarities and differences between methods developed independently by the dynamical systems, fluid dynamics, and molecular dynamics communities such as \emph{time-lagged independent component analysis} (TICA), \emph{dynamic mode decomposition} (DMD), and their respective generalizations. As a result, extensions and best practices developed for one particular method can be carried over to other related methods.
\end{abstract}

\section{Introduction}

The numerical solution of complex systems of differential equations plays an important role in many areas such as molecular dynamics, fluid dynamics, mechanical as well as electrical engineering, and physics. These systems often exhibit multi-scale behavior which can be due to the coupling of subsystems with different time scales \textendash{} for instance, fast electrical and slow mechanical components \textendash{} or due to the intrinsic properties of the system itself \textendash{} for instance, the fast vibrations and slow conformational changes of molecules. Analyzing such problems using transfer operator based methods is often infeasible or prohibitively expensive from a computational point of view due to the so-called \emph{curse of dimensionality}. One possibility to avoid this is to project the dynamics of the high-dimensional system onto a lower-dimensional space and to then analyze the reduced system representing, for instance, only the relevant slow dynamics, see, e.g., \cite{PPGDN13,FGH14a}. 

In this paper, we will introduce different methods such as \emph{time-lagged independent component analysis} (TICA) \cite{MS94,PPGDN13,SP13} and \emph{dynamic mode decomposition} (DMD) \cite{Schmid10,CTR12,TRLBK14,KBBP16} to identify the dominant dynamics using only simulation data or experimental data. It was shown that these methods are related to Koopman operator approximation techniques \cite{Ko31,Mezic05,RMBSH09,BMM12}. Extensions of the aforementioned methods called the \emph{variational approach of conformation dynamics} (VAC) \cite{NoNu13,NKPMN14,NSVN15} developed mainly for reversible molecular dynamics problems and \emph{extended dynamic mode decomposition} (EDMD) \cite{WKR15,WRK15,KKS16} can be used to compute eigenfunctions, eigenvalues, and eigenmodes of the Koopman operator (and its adjoint, the Perron\textendash Frobenius operator). Interestingly, although the underlying ideas, derivations, and intended applications of these methods differ, the resulting algorithms share a lot of similarities. The goal of this paper is to show the equivalence of different data-driven methods which have been widely used by the dynamical systems, fluid dynamics, and molecular dynamics communities, but under different names. Hence, extensions, generalizations, and algorithms developed for one method can be carried over to its counterparts, resulting in a unified theory and set of tools. An alternative approach to data-driven model reduction \textendash{} also related to transfer operators and their generators \textendash{} would be to use \emph{diffusion maps}
\cite{NLCK06,CKLMN08,FPKD11,Gia15}. Manifold learning methods, however, are beyond the scope of this paper.

The outline of the paper is as follows: Section~\ref{sec:Transfer operators and reversibility} briefly introduces transfer operators and the concept of reversibility. In Section~\ref{sec:Numerical approximation of transfer operators}, different data-driven methods for the approximation of the eigenvalues, eigenfunctions, and eigenmodes of transfer operators will be described. The theoretical background and the derivation of these methods will
be outlined in Section~\ref{sec:Derivation}. Section~\ref{sec:Conclusion} addresses open problems and lists possible future work. 

\section{Transfer operators and reversibility\label{sec:Transfer operators and reversibility}}

In the literature, the term \emph{transfer operator} is sometimes used in different contexts. In this section, we will briefly introduce the Perron\textendash Frobenius operator, the Perron\textendash Frobenius operator with respect to the equilibrium density, and the Koopman operator. All these three operators are, according to our definition, transfer operators.

\subsection{Guiding example}

Our paper will deal with data-driven methods to analyze both stochastic and deterministic dynamical systems. To illustrate the concepts of transfer operators and their spectral components, we first introduce a simple stochastic dynamical system that will be revisited throughout the paper.

\begin{example}
\label{ex:OU process}Consider the following one-dimensional Ornstein\textendash Uhlenbeck process, given by an Itô stochastic differential equation\footnote{A general time-homogeneous Itô stochastic differential equation is given by $\mathrm{d}\mathbf{X}_{t}=-\alpha(\mathbf{X}_{t})\,\mathbf{X}_{t}\,\mathrm{d}t+\sigma(\mathbf{X}_{t})\,\mathrm{d}\mathbf{W}_{t}$, where $\alpha:\mathbb{R}^{d}\to\mathbb{R}^{d}$ and $\sigma:\mathbb{R}^{d}\to\mathbb{R}^{d\times d}$ are coefficient functions, and $\{\mathbf{W}_{t}\}_{t\ge0}$ is a $d$-dimensional standard Wiener process.} of the form:
\begin{equation*}
    \mathrm{d}\boldsymbol{X}_{t}=-\alpha D\,\boldsymbol{X}_{t}\,\mathrm{d}t+\sqrt{2D}\,\mathrm{d}\boldsymbol{W}_{t}.
\end{equation*}
Here, $\{W_{t}\}_{t\ge0}$ is a one-dimensional standard Wiener process (Brownian motion), the parameter $\alpha$ is the friction coefficient, and $D=\beta^{-1}$ is the diffusion coefficient. The stochastic forcing
usually models physical effects, most often thermal fluctuations and it is customary to call $\beta$ the inverse temperature.

The transition density of the Ornstein\textendash Uhlenbeck process, i.e., the conditional probability density to find the process near $y$ a time $\tau$ after it had been at $x$, is given by
\begin{equation}
    p_{\tau}(x,y)=\frac{1}{\sqrt{2\,\pi\,\sigma^{2}(\tau)}}\exp\left(-\frac{\left(y-x\,e^{-\alpha D\tau}\right)^{2}}{2\,\sigma^{2}(\tau)}\right),\label{eq:transition_density_OU}
\end{equation}
where $\sigma^{2}(\tau)=\alpha^{-1}\left(1-e^{-2\alpha D\tau}\right)$. Figure \ref{fig:Density OU}a shows the transition densities for different values of $\tau$. More details can be found in~\cite{Pav14}. For complex dynamical systems, the transition density is not known explicitly, but must be estimated from simulation or measurement data. 

\begin{figure}[tbh]
    \centering
    \subfiguretitle{a)}
    \includegraphics[width=0.9\columnwidth]{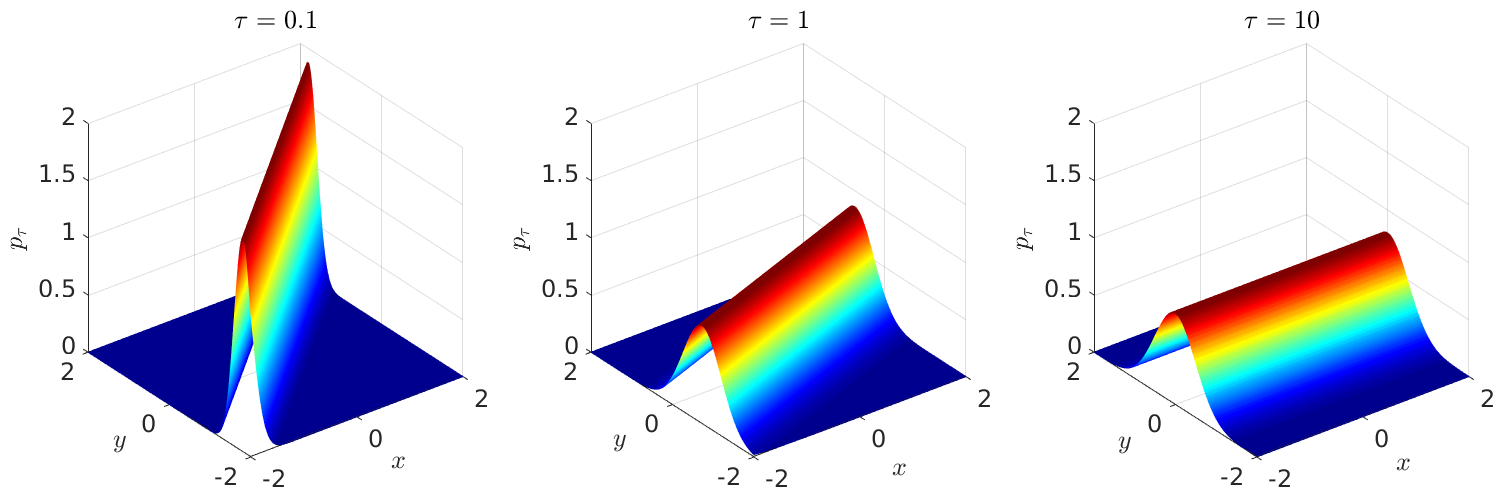} \\
    \subfiguretitle{b)}
    \includegraphics[width=0.9\columnwidth]{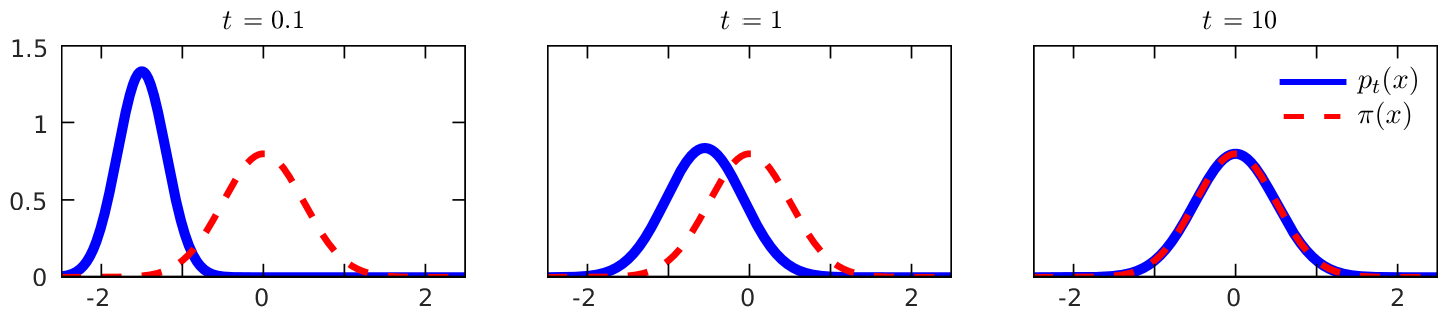}
    \caption{a) Transition density function of the Ornstein\textendash Uhlenbeck process for different values of $\tau$. If $\tau$ is small, values starting in $x$ will stay close to it. For larger values of $\tau$, the influence of the starting point $x$ is negligible. Densities converge to the equilibrium density, denoted by $\pi$. Here, $\alpha=4$ and $D=0.25$. b) Evolution of the probability to find the dynamical system at any point $x$ over time $t$, after starting with a peaked distribution at $t=0$. We show the resulting distributions at times $t=0.1$, and $t=1$, and $t=10$. The system relaxes towards the stationary density $\pi(x)$.}
    \label{fig:Density OU}
\end{figure}

In this work we will describe the dynamics of a system in terms of dynamical operators such as the propagator $\mathcal{P}_{\tau}$, which is defined by the transition density $p_{\tau}(x,y)$ and propagates a probability density of Brownian walkers in time by
\begin{equation*}
    p_{t+\tau}(x)=\mathcal{P}_{\tau}\,p_{t}(x).
\end{equation*}
See Figure~\ref{fig:Density OU}b for the time evolution of the Ornstein\textendash Uhlenbeck process initiated from a localized starting condition. It can be seen that the distribution spreads out and converges towards a Gaussian distribution, which is then invariant in time. For this simple dynamical system we can give the equation for the invariant density explicitly:
\begin{equation} \label{eq:stationary_density_OU}
    \pi(x)=\frac{1}{\sqrt{2\,\pi\,\alpha^{-1}}}\exp\left(-\frac{x^{2}}{2\,\alpha^{-1}}\right),
\end{equation}
which is a Gaussian whose variance is decreasing with increasing friction and decreasing temperature. \exampleSymbol
\end{example}

\subsection{Transfer operators}

Let $\{\mathbf{X}_{t}\}_{t\ge0}$ be a time-homogeneous\footnote{We call a stochastic process $\{\mathbf{X}_{t}\}_{t\ge0}$ \emph{time-homogeneous}, or \emph{autonomous}, if it holds for every $t\ge s\ge0$ that the distribution of $\mathbf{{X}}_{t}$ conditional to $\mathbf{X}_{s}=x$ only depends on $x$ and $(t-s)$. It is the stochastic analogue of the flow of an autonomous (time-independent) ordinary differential equation.} stochastic process defined on the bounded state space $\mathbb{X}\subset\mathbb{R}^{d}$. It can be genuinely stochastic or it might as well be deterministic, such that there is a flow map $\Phi_{\tau}:\mathbb{X}\to\mathbb{X}$ with $\Phi_{\tau}(\mathbf{X}_{t})=\mathbf{X}_{t+\tau}$ for $\tau\ge0$. Let the measure\footnote{For a measure-theoretic discussion of this construction, please refer to \cite{KKS16}. For our purposes, it is sufficient to equip $\mathbb{X}$ with the standard Lebesgue measure. In particular, if not stated otherwise, measurability of a set $\mathbb{A\subset X}$ is meant with respect to the Borel $\sigma$-algebra.} $\mathbb{\mathbb{P}}$ denote the law of the process $\{\mathbf{X}_{t}\}_{t\ge0}$ that we will study in terms of its statistical transition properties. To this end, under some mild regularity assumptions\footnote{These conditions are called interchangeably \emph{absolute continuity},$\mu$-\emph{compatibility}, or \emph{null preservingness}.} which are satisfied by Itô diffusions with smooth coefficients \cite{Hopf54,Kre85}, we can give the following definition.

\begin{defn}
The transition density function $p_{\tau}:\mathbb{X}\times\mathbb{X}\to[0,\,\infty]$
of a process $\{\mathbf{X}_{t}\}_{t\ge0}$ is defined by
\begin{equation*}
    \mathbb{P}[\mathbf{X}_{t+\tau}\in\mathbb{A}\,\vert\,\mathbf{X}_{t}=x]=\intop_{\mathbb{A}}p_{\tau}(x,y)\,\mathrm{d}y,
\end{equation*}
for every measurable set $\mathbb{A}.$ Here and in what follows, $\mathbb{P}[\,\cdot\,\vert\,\mathfrak{E}]$ denotes probabilities conditioned on the event $\mathfrak{E}$. That is, $p_{\tau}(x,y)$ is the conditional probability density of $\mathbf{X}_{t+\tau}=y$ given that $\mathbf{X}_{t}=x$.
\end{defn}

For $1\le r\le\infty$, the spaces $L^{r}(\mathbb{X}$) denote the usual spaces of $r$-Lebesgue integrable functions, which is a Banach space with the corresponding norm $\|\cdot\|_{L^{r}}$.

\begin{defn}
Let $p_{t}\in L^{1}(\mathbb{X})$ be the probability density and $f_{t}\in L^{\infty}(\mathbb{X})$
an observable of the system. For a given lag time $\tau$: 
\begin{enumerate}
\item[a)] \setlength\belowdisplayskip{0pt}The \emph{Perron\textendash Frobenius
operator} or \emph{propagator} $\mathcal{P}_{\tau}:L^{1}(\mathbb{X})\to L^{1}(\mathbb{X})$
is defined by 
\begin{equation*}
    \mathcal{P}_{\tau}p_{t}(x)=\intop_{\mathbb{X}}p_{\tau}(y,x)\,p_{t}(y)\,\mathrm{d}y.
\end{equation*}
\item[b)] The \emph{Koopman operator} $\mathcal{K}_{\tau}:L^{\infty}(\mathbb{X})\to L^{\infty}(\mathbb{X})$
is defined by
\begin{equation*}
    \mathcal{K}_{\tau}f_{t}(x)=\intop_{\mathbb{X}}p_{\tau}(x,y)\,f_{t}(y)\,\mathrm{d}y=\mathbb{E}[f_{t}(\mathbf{X}_{t+\tau})\,\vert\,\mathbf{X}_{t}=x].
\end{equation*}
\end{enumerate}
\end{defn}

Both $\mathcal{P}_{\tau}$ and $\mathcal{K}_{\tau}$ are linear but infinite-dimensional operators which are adjoint to each other with respect to the standard duality pairing $\langle\cdot,\,\cdot\rangle$, defined by $\langle f,\,g\rangle=\int_{\mathbb{X}}f(x)\,g(x)\,dx$. The homogeneity of the stochastic process $\{\mathbf{X}_{t}\}_{t\ge0}$ implies the so-called \emph{semigroup property} of the operators,
i.e., $\mathcal{P}_{\tau+\sigma}=\mathcal{P_{\tau}}\mathcal{P}_{\sigma}$ and $\mathcal{K}_{\tau+\sigma}=\mathcal{K_{\tau}}\mathcal{K_{\sigma}}$ for $\tau,\sigma\ge0$. In other words, these operators describe time-stationary Markovian dynamics. While the Perron\textendash Frobenius operator describes the evolution of densities, the Koopman operator describes the evolution of observables. For the analysis of the long-term behavior of dynamical systems, densities that remain unchanged by the dynamics
play an important role (one can think of the concept of \textit{ergodicity}).

\begin{defn}
A density $\pi$ is called an \emph{invariant density} or \emph{equilibrium density} if $\mathcal{P}_{\tau}\,\pi=\pi$. That is, the equilibrium density $\pi$ is an eigenfunction of the Perron\textendash Frobenius
operator $\mathcal{P}_{\tau}$ with corresponding eigenvalue $1$.
\end{defn}

In what follows, $L_{\pi}^{r}(\mathbb{X})$ denotes the weighted $L^{r}$-space of functions $f$ such that $\|f\|_{L_{\pi}^{r}}:=\int_{\mathbb{X}}|f(x)|^{r}\pi(x)\,dx<\infty$. While one can consider the evolution of densities with respect to any density, we are particularly interested in the evolution with respect to the equilibrium density. From this point on, we assume there is a unique invariant density. This assumption is typically satisfied for molecular dynamics applications, where the invariant density is given by the Boltzmann distribution.

\begin{defn}
Let $L_{\pi}^{1}(\mathbb{X})\ni u_{t}(x)=\pi(x)^{-1}\,p_{t}(x)$ be a probability density with respect to the equilibrium density $\pi$. Then the \emph{Perron\textendash Frobenius operator (propagator) with
respect to the equilibrium density}, denoted by $ \mathcal{T}_\tau $, is defined by
\begin{equation*}
    \mathcal{T}_{\tau}u_{t}(x) = \intop_{\mathbb{X}}\frac{\pi(y)}{\pi(x)}\,p_{\tau}(y,x)\,u_{t}(y)\,\mathrm{d}y.
\end{equation*}
\end{defn}

The operators $\mathcal{P}_{\tau}$ and $\mathcal{K}_{\tau}$ can be defined on other spaces~$ L^r $ and $ L^{r'} $, with~$ r \ne 1 $ and~$ r' \ne \infty $, see \cite{BaRo95, KKS16} for more details. By defining the weighted duality pairing $\langle f,g\rangle_{\pi}=\int_{\mathbb{X}}f(x)\,g(x)\,\pi(x)\,dx$ for $f\in L_{\pi}^{r}(\mathbb{X})$ and $g\in L_{\pi}^{r'}(\mathbb{X})$, where $\frac{1}{r}+\frac{1}{r'}=1$, $\mathcal{T}_{\tau}$ defined on $L_{\pi}^{r'}(\mathbb{X})$ is the adjoint of $\mathcal{K}_{\tau}$ defined on $L_{\pi}^{r}(\mathbb{X})$ with respect to $\langle\cdot,\,\cdot\rangle_{\pi}$:
\begin{equation*}
    \langle\mathcal{K}_{\tau}f,\,g\rangle_{\pi} = \langle f,\,\mathcal{T}_{\tau}g\rangle_{\pi}.
\end{equation*}
For more details, see~\cite{LaMa94,SFHD99,FrJuKo13,NoNu13,NKPMN14,KKS16,WNPKKN16}. The two operators $\mathcal{P}_{\tau}$ and $\mathcal{T}_{\tau}$ are often referred to as \emph{forward operators}, whereas $\mathcal{K}_{\tau}$ is also called \emph{backward operator}, as they are the solution operators of the forward (Fokker\textendash Planck) and backward Kolmogorov equations \cite[Section 11]{LaMa94}, respectively.

\subsection{Spectral decomposition of transfer operators}

In what follows, let $\mathcal{A}_{\tau}$ denote one of the transfer operators defined above, i.e., $\mathcal{P}_{\tau}$, $\mathcal{T}_{\tau}$, or $\mathcal{K}_{\tau}$. We are particularly interested in computing eigenvalues $\lambda_{\ell}(\tau) \in \mathbb{C}$ and eigenfunctions $\varphi_{\ell}:\mathbb{X}\to\mathbb{C}$ of transfer operators, i.e.:
\begin{equation*}
    \mathcal{A}_{\tau}\varphi_{\ell} = \lambda_{\ell}(\tau) \,\varphi_{\ell}.
\end{equation*}
Note that the eigenvalues depend on the lag time $ \tau $. For the sake of simplicity, we will often omit this dependency. The eigenvalues and eigenfunctions of transfer operators contain important information about the global properties of the system such as metastable sets or fast and slow processes and can also be used as reduced coordinates, see~\cite{DJ99,SFHD99,Mezic05,SS13,FrJuKo13,FGH14a} and references therein. 

\begin{example}
\label{ex:OU eigenvalues}The eigenvalues $\lambda_{\ell}$ and eigenfunctions
$\varphi_{\ell}$ of $\mathcal{K_{\tau}}$ associated with the Ornstein--Uhlenbeck process introduced in Example~\ref{ex:OU process} are given by 
\begin{equation*}
    \lambda_{\ell}(\tau)=e^{-\alpha D\,(\ell-1)\,\tau},\quad\varphi_{\ell}(x)=\frac{1}{\sqrt{(\ell-1)!}}\,H_{\ell-1}\left(\sqrt{\alpha}\,x\right),\quad\ell=1,2,\dots,
\end{equation*}
where $H_{\ell}$ denotes the $\ell$th probabilists' Hermite polynomial \cite{Pav14}. The eigenfunctions of $\mathcal{P}_{\tau}$ are given by the eigenfunctions of $\mathcal{K}_{\tau}$ multiplied by the equilibrium density $\pi$, see also Figure~\ref{fig:Eigenfunctions OU}. \exampleSymbol

\begin{figure}
    \centering
    \begin{minipage}[t]{0.49\columnwidth}%
        \centering
        \subfiguretitle{a)}\includegraphics[width=0.9\textwidth]{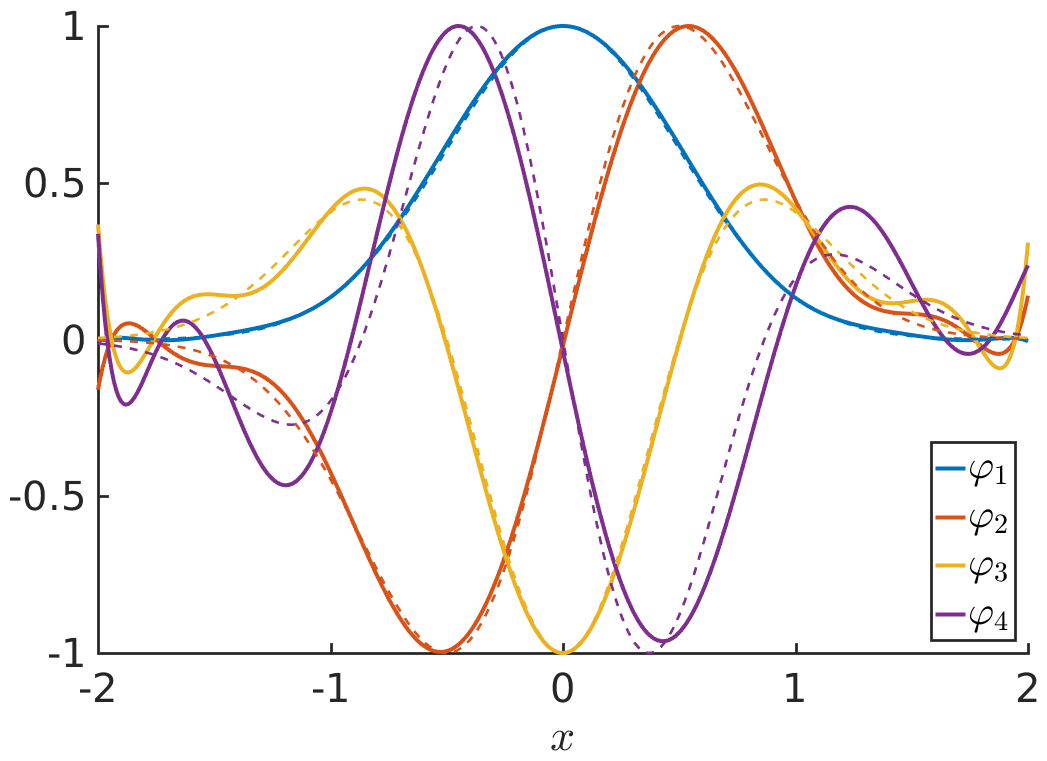}
    \end{minipage}%
    \begin{minipage}[t]{0.49\columnwidth}%
        \centering
        \subfiguretitle{b)}\includegraphics[width=0.9\textwidth]{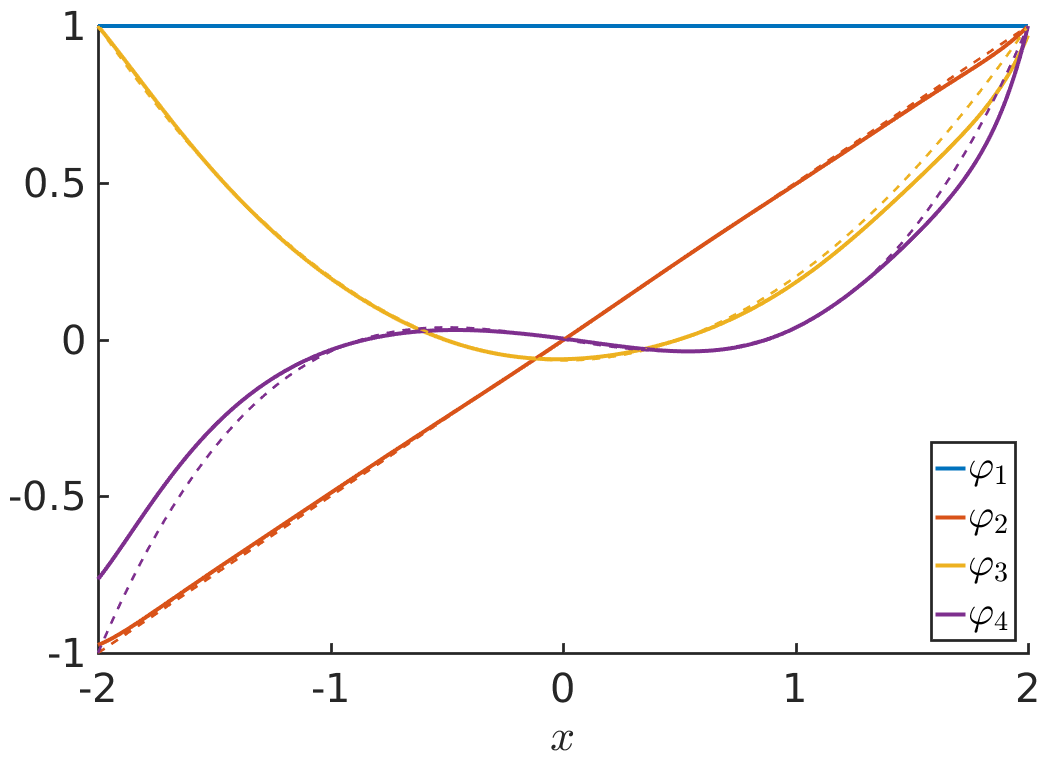}
    \end{minipage}
    \caption{Dominant eigenfunctions of the Ornstein\textendash Uhlenbeck process computed analytically (dotted lines) and using VAC/EDMD (solid lines). a) Eigenfunctions of the Perron\textendash Frobenius operator $\mathcal{P}_{\tau}$. b) Eigenfunctions of the Koopman operator $\mathcal{K_{\tau}}$.}
    \label{fig:Eigenfunctions OU}
\end{figure}
\end{example}

In addition to the eigenvalues and eigenfunctions, an essential part of the Koopman operator analysis is the set of \emph{Koopman modes} for the so-called \emph{full-state observable }$g(x)=x$. The Koopman modes are vectors that, together with the eigenvalues and eigenfunctions, allow us to reconstruct and to propagate the system's state~\cite{WKR15}. More precisely, assume that each component $g_{i}$ of the full-state observable, i.e., $g_{i}(x)=x_{i}$ for $i=1,\,\dots,\,d$, can be written in terms of the eigenfunctions as $g_{i}(x)=\sum_{\ell}\varphi_{\ell}(x)\,\eta_{i\ell}$. Defining the Koopman modes by $\eta_{\ell}=[\eta_{1\ell},\dots,\,\eta_{d\ell}]^{T}$, we obtain $g(x) = x=\sum_{\ell}\varphi_{\ell}(x)\,\eta_{\ell}$ and thus 
\begin{equation} \label{eq:Koopman reconstruction}
    \mathcal{K}_{\tau}g(x)
        = \mathbb{E}[g(\mathbf{X}_{\tau})\,\vert\,\mathbf{X}_{0}=x]
        = \mathbb{E}[\mathbf{X}_{\tau}\,\vert\,\mathbf{X}_{0}=x]
        = \sum_{\ell}\lambda_{\ell}(\tau)\,\varphi_{\ell}(x)\,\eta_{\ell}.
\end{equation}
For vector-valued functions, the Koopman operator is defined to act componentwise. In order to be able to compute eigenvalues, eigenfunctions, and eigenmodes numerically, we project the infinite-dimensional operators onto finite-dimensional spaces spanned by a given set of basis functions. This will be described in detail in Section~\ref{sec:Derivation}.

\subsection{Reversibility}

We briefly recapitulate the properties of reversible systems. For many applications, including commonly used molecular dynamics models, the dynamics in full phase space are known to be reversible.

\begin{defn}
A system is said to be reversible if the so-called detailed balance condition is fulfilled, i.e., it holds for all $x,y\in\mathbb{X}$ that 
\begin{equation}
    \pi(x)\,p_{\tau}(x,y)=\pi(y)\,p_{\tau}(y,x).\label{eq:detailed_balance}
\end{equation}
\end{defn}

\begin{example}
The Ornstein\textendash Uhlenbeck process is reversible. It is straightforward to verify that \eqref{eq:detailed_balance} is fulfilled by the transition density \eqref{eq:transition_density_OU} and the stationary density
\eqref{eq:stationary_density_OU} for all values of $x$, $y$ and $\tau$. Also general Smoluchowski equations of a $d$-dimensional system of the form 
\begin{equation*}
    \mathrm{d}\mathbf{X}_{t } =-D\nabla V(\mathbf{X}_{t})\,\mathrm{d}t+\sqrt{2dD}\,\mathrm{d}\mathbf{W}_{t}
\end{equation*}
with dimensionless potential $V(x)$ are reversible. The stationary density is then given by $\pi\varpropto\exp(-V(x))$ \cite{Lei06}.\exampleSymbol
\end{example}

As a result of the detailed balance condition, the Koopman operator $\mathcal{K}_{\tau}$ and the Perron\textendash Frobenius operator with respect to the equilibrium density, $\mathcal{T}_{\tau}$, are identical (hence also self-adjoint):
\begin{equation*}
    \mathcal{K_{\tau}}f=\intop_{\mathbb{X}}p_{\tau}(x,y)\,f(y)\,\mathrm{d}y=\intop_{\mathbb{X}}\frac{\pi(y)}{\pi(x)}\,p_{\tau}(y,x)\,f(y)\,\mathrm{d}y=\mathcal{T}_{\tau}\,f.
\end{equation*}
Moreover, both $\mathcal{K}_{\tau}$ and $\mathcal{P}_{\tau}$ become self-adjoint with respect to the stationary density, i.e.
\begin{align*}
    \langle\mathcal{P}_{\tau}f,\,g\rangle_{\pi^{-1}} & =\langle f,\mathcal{\,P}_{\tau}g\rangle_{\pi^{-1}},\\
    \langle\mathcal{K_{\tau}}f,\,g\rangle_{\pi} & =\langle f,\,\mathcal{K}_{\tau}g\rangle_{\pi}.
\end{align*}
Hence, the eigenvalues $\lambda_{\ell}$ are real and the eigenfunctions $\varphi_{\ell}$ of $\mathcal{K}_{\tau}$ form an orthogonal basis with respect to $\langle\cdot,\,\cdot\rangle_{\pi}$. That is, the eigenfunctions can be scaled so that $\langle\varphi_{\ell},\,\varphi_{\ell'}\rangle_{\pi}=\delta_{\ell\ell'}$. Furthermore, the leading eigenvalue $\lambda_{1}$ is the only eigenvalue with absolute value $1$ and we obtain
\begin{equation*}
    1=\lambda_{1}>\lambda_{2}\ge\lambda_{3}\ge\dots,
\end{equation*}
see, e.g., \cite{NKPMN14}. We can then expand a function $f\in L_{\pi}^{2}(\mathbb{X})$ in terms of the eigenfunctions as $f=\sum_{\ell=1}^{\infty}\langle f,\,\varphi_{\ell}\rangle_{\pi}\,\varphi_{\ell}$ such that
\begin{equation}
    \mathcal{K}_{\tau}f=\sum_{\ell=1}^{\infty}\lambda_{\ell}(\tau)\,\langle f,\,\varphi_{\ell}\rangle_{\pi}\,\varphi_{\ell}.\label{eq:spectral_decomposition_rev}
\end{equation}
Furthermore, the eigenvalues decay exponentially with $\lambda_{\ell}(\tau)=\exp(-\kappa_{\ell}\tau)$ with relaxation rate $\kappa_{\ell}$ and relaxation timescale $t_{\ell}^{-1}$. Thus, for a sufficiently large lag time $\tau$, the fast relaxation processes have decayed and \eqref{eq:spectral_decomposition_rev} can be approximated by finitely many terms. The propagator $\mathcal{P}_{\tau}$ has the same eigenvalues and the eigenfunctions $\tilde{\varphi}_{\ell}$ are given by $\tilde{\varphi}_{\ell}(x)=\pi(x)\,\varphi_{\ell}(x)$.

\section{Data-driven approximation of transfer operators\label{sec:Numerical approximation of transfer operators}}

In this section, we will describe different data-driven methods to identify the dominant dynamics of dynamical systems and to compute eigenfunctions of transfer operators associated with the system, namely
TICA and DMD as well as VAC and EDMD. A formal derivation of methods to compute finite-dimensional approximations of transfer operators \textendash{} resulting in the aforementioned methods \textendash{}
will be given in Section~\ref{sec:Derivation}. Although TICA can be regarded as a special case of VAC, and DMD as a special case of EDMD, these methods are often used in different settings. With the aid of TICA, for instance, it is possible to identify the main slow coordinates and to project the dynamics onto the resulting reduced space, which can then be discretized using conventional Markov state models (a special case of VAC or EDMD, respectively, see Subsection~\ref{subsec:Relationships with other methods}). We will introduce the original methods \textendash{} TICA and DMD~\textendash{} first and then extend these methods to the more general case. Since in many publications a different notation is used, we will first start with the required basic definitions. 

In what follows, let $x_{i},y_{i}\in\mathbb{R}^{d}$, $i=1,\dots,m$, be a set of pairs of $d$-dimensional data vectors, where $x_{i}=\mathbf{X}_{t_{i}}$ and $y_{i}=\mathbf{X}_{t_{i}+\tau}$. Here, the underlying dynamical
system is not necessarily known, the vectors $x_{i}$ and $y_{i}$ can simply be measurement data or data from a black-box simulation. In matrix form, this can be written as 
\begin{equation}
    X=\begin{bmatrix}x_{1} & x_{2} & \cdots & x_{m}\end{bmatrix}\quad\text{and}\quad Y=\begin{bmatrix}y_{1} & y_{2} & \cdots & y_{m}\end{bmatrix},\label{eq:DMD data}
\end{equation}
with $X,\,Y\in\mathbb{R}^{d\times m}$. If one long trajectory $\{z_{0},\,z_{1},\,z_{2},\,\dots\}$ of a dynamical system is given, i.e., $z_{i}=\mathbf{X}_{t_{0}+h\,i}$, where $h$ is the step size and $\tau=n_{\tau}\,h$ the lag time, we obtain
\begin{equation*}
    X=\begin{bmatrix}z_{0} & z_{1} & \cdots & z_{m-1}\end{bmatrix}\quad\text{and}\quad Y=\begin{bmatrix}z_{n_{\tau}} & z_{n_{\tau}+1} & \cdots & z_{n_{\tau}+m-1}\end{bmatrix}.
\end{equation*}
That is, in this case $Y$ is simply $X$ shifted by the lag time $\tau$. Naturally, if more than one trajectory is given, the data matrices $X$ and $Y$ can be concatenated.

In addition to the data, VAC and EDMD require a set of uniformly bounded basis functions or observables, given by~$\{\psi_{1},\,\psi_{2},\,\dots,\,\psi_{k}\}\subset L^{\infty}(\mathbb{X})$. Since~$\mathbb{X}$ is assumed to be bounded, we have~$\psi_i\in L^r(\mathbb{X})$ for all~$i=1,\ldots,k$ and~$1\le r\le \infty$. The basis functions could, for instance, be monomials, indicator functions,
radial basis functions, or trigonometric functions. The optimal choice of basis functions remains an open problem and depends strongly on the system. If the set of basis functions is not sufficient to represent the eigenfunctions, the results will be inaccurate. A too large set of basis functions, on the other hand, might lead to ill-conditioned matrices and overfitting. Cross-validation strategies have been developed to detect overfitting~\cite{MP15}.

For a basis $\psi_{i}$, $i=1,\,\dots,\,k$, define $\psi:\mathbb{R}^{d}\to\mathbb{R}^{k}$ to be the vector-valued function given by
\begin{equation}
    \psi(x) = [\psi_{1}(x),\,\psi_{2}(x),\,\dots,\,\psi_{k}(x)]^{T}.\label{eq:Basis functions}
\end{equation}
The goal then is to find the best approximation of a given transfer operator in the space spanned by these basis functions. This will be explained in detail in Section \ref{sec:Derivation}. In addition
to the data matrices $X$ and $Y$, we will need the transformed data matrices
\begin{equation} \label{eq:PsiX PsiY}
    \Psi_{X} = 
    \begin{bmatrix} \psi(x_{1}) & \psi(x_{2}) & \dots & \psi(x_{m}) \end{bmatrix}
    \quad \text{and} \quad
    \Psi_{Y} =
    \begin{bmatrix} \psi(y_{1}) & \psi(y_{2}) & \dots & \psi(y_{m}) \end{bmatrix}.
\end{equation}

\subsection{Time-lagged independent component analysis}

\emph{Time-lagged independent component analysis} (TICA) has been introduced in \cite{MS94} as a solution to the \emph{blind source separation} problem, where the correlation matrix and the time-delayed
correlation matrix are used to separate superimposed signals. The term TICA has been introduced later \cite{HyvaerinenKarhunenOja_ICA_Book}. TDSEP \cite{ZieheMueller_ICANN98_TDSEP}, an extension of TICA, is
popular in the machine learning community. It was shown only recently that TICA is a special case of the VAC by computing the optimal linear projection for approximating the slowest relaxation processes, and as such provides an approximation of the leading eigenvalues and eigenfunctions of transfer operators \cite{PPGDN13}. TICA is now a popular dimension reduction technique in the field of molecular dynamics \cite{PPGDN13,SP13}.
That is, TICA is used as a preprocessing step to reduce the size of the state space by projecting the dynamics onto the main coordinates. The time-lagged independent components are required \textbf{(a)} to be uncorrelated and \textbf{(b)} to maximize the autocovariances at lag time $\tau$, see \cite{HyvaerinenKarhunenOja_ICA_Book,PPGDN13} for more details. Assuming that the system is reversible, the TICA coordinates are the eigenfunctions of $\mathcal{T}_{\tau}$ or $\mathcal{K}_{\tau}$, respectively, projected onto the space spanned by linear basis functions,~i.e., $\psi(x)=x$.

Let $C(\tau)$ be the time-lagged covariance matrix defined by 
\begin{equation*}
    C_{ij}(\tau) = \langle\mathbf{X}_{t,i}\,\mathbf{X}_{t+\tau,j}\rangle_{t} = \mathbb{E}_{\pi}\left[\mathbf{X}_{t,i}\,\mathbf{X}_{t+\tau,j}\right].
\end{equation*}
Given data $X$ and $Y$ as defined above, estimators $C_{0}$ and $C_{\tau}$ for the covariance matrices $C(0)$ and $C(\tau)$ can be computed as
\begin{equation}
    \begin{split}
        C_{0} & =\tfrac{1}{m-1}\sum_{k=1}^{m}x_{k}\,x_{k}^{T}=\tfrac{1}{m-1}XX^{T},\\
        C_{\tau} & =\tfrac{1}{m-1}\sum_{k=1}^{m}x_{k}\,y_{k}^{T}=\tfrac{1}{m-1}XY^{T}.
    \end{split}
    \label{eq:C_O and C_tau TICA}
\end{equation}
The time-lagged independent components are defined to be solutions of the eigenvalue problem
\begin{equation}
    C_{\tau}\,\xi_{\ell}=\lambda_{\ell}\,C_{0}\,\xi_{\ell}\quad\text{or}\quad C_{0}^{+}C_{\tau}\,\xi_{\ell}=\lambda_{\ell}\,\xi_{\ell},\label{eq:TICA eigenvectors}
\end{equation}
respectively. In what follows, let $M_{\tsub{TICA}}=C_{0}^{+}C_{\tau}$,
where $C_{0}^{+}$ denotes the Moore\textendash Penrose pseudo-inverse
of $C_{0}$.

In applications, often the symmetrized estimators 
\begin{equation*}
    C_{0} = \tfrac{1}{2m-2}(XX^{T}+YY^{T}) \quad \text{and} \quad C_{\tau}=\tfrac{1}{2m-2}(XY^{T}+YX^{T})
\end{equation*}
are used so that the resulting TICA coordinates become real-valued. This corresponds to averaging over the trajectory and the time-reversed trajectory. Note that this symmetrization can introduce a large estimator bias that affects the dominant spectrum of \eqref{eq:TICA eigenvectors}, if the process is non-stationary, or the distribution of the data is far from the equilibrium of the process. In the latter case, a reweighting procedure can be applied to obtain weighted versions of the estimators \eqref{eq:C_O and C_tau TICA}, to reduce that bias \cite{WNPKKN16}.

\begin{example}
Let us illustrate the idea behind TICA with a simple example. Consider the data shown in Figure~\ref{fig:PCA and TICA}, which was generated by a stochastic process which will typically spend a long time in one of the two clusters before it jumps to the other. We are interested in finding these metastable sets. Standard \emph{principal component analysis} (PCA) leads to the coordinate shown in red, whereas TICA \textendash{} shown in black \textendash{} takes time-information into account and is thus able to identify the slow direction of the system correctly. Projecting the system onto the $x$-coordinate will preserve the slow process while eliminating the fast stochastic noise.\exampleSymbol

\begin{figure}[tbh]  
    \centering
    \includegraphics[width=0.95\textwidth]{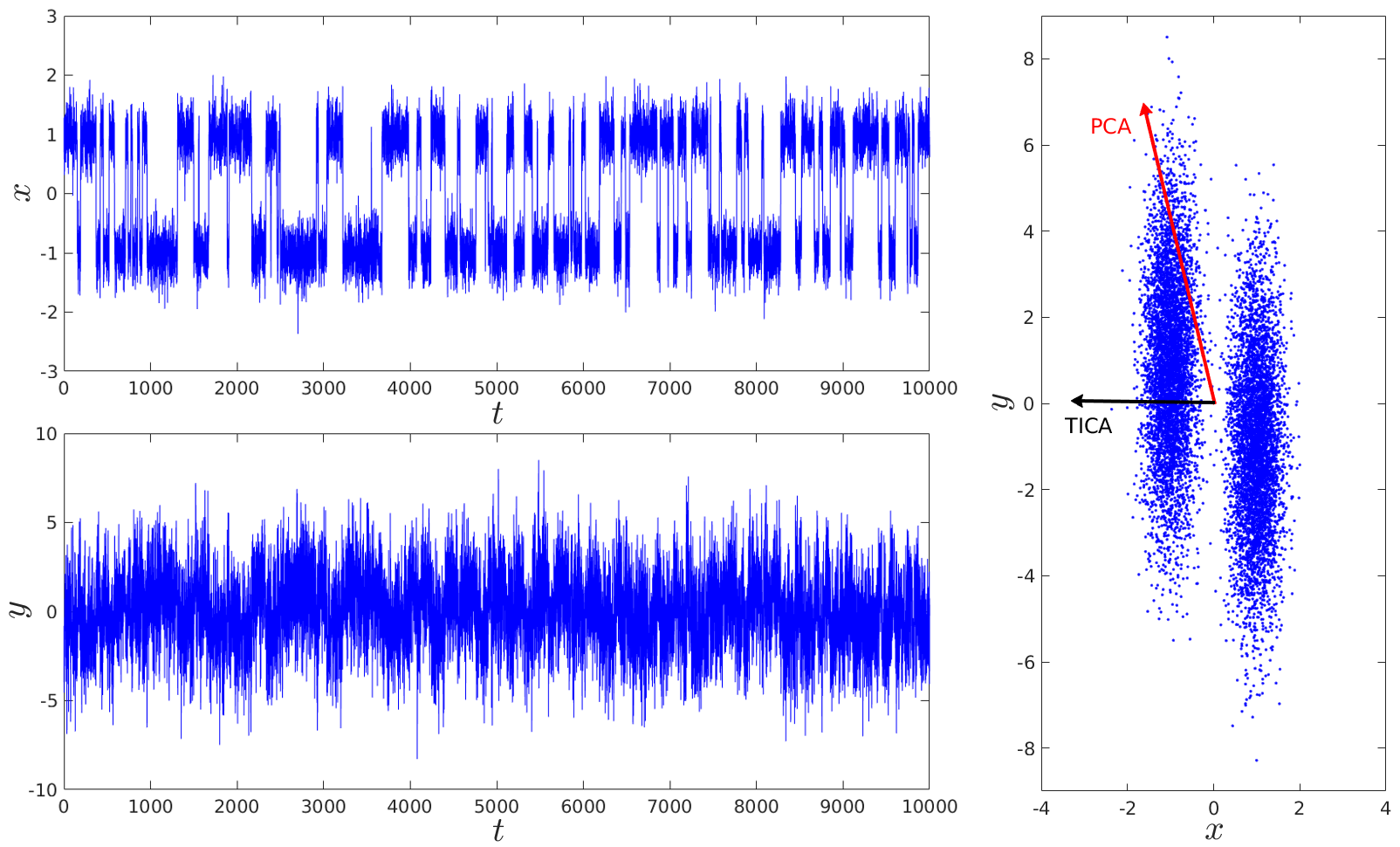}
    \caption{The difference between PCA and TICA. The top and bottom plot show the $x$- and $y$-component of the system, respectively, the plot on the right the resulting main principal component vector and the main TICA coordinate.}
    \label{fig:PCA and TICA}
\end{figure}
\end{example}

\begin{algorithm}
\begin{enumerate}
\item \setlength{\itemsep}{0mm}Compute a reduced SVD of $X$, i.e., $X=U\,\Sigma\,V^{T}$.
\item Whiten data: $\tilde{X}=\Sigma^{-1}U^{T}X$ and $\tilde{Y}=\Sigma^{-1}U^{T}Y$.
\item Compute $\overline{M}_{\tsub{TICA}}=\tilde{X}\tilde{Y}^{T}=\Sigma^{-1}U^{T}XY^{T}U\Sigma^{-1}$.
\item Solve the eigenvalue problem $\overline{M}_{\tsub{TICA}}\,w_{\ell}=\lambda_{\ell}\,w_{\ell}$.
\item The TICA coordinates are then given by $\xi_{\ell}=U\Sigma^{-1}w_{\ell}$.
\end{enumerate}
\vspace*{-0.5\baselineskip}\caption{AMUSE algorithm to compute TICA.}
\label{alg:AMUSE}
\end{algorithm}

The TICA coordinates can be computed using AMUSE\footnote{Algorithm for Multiple Unknown Signals Extraction.} \cite{TSHL90} as shown in Algorithm \ref{alg:AMUSE}. Instead of computing a singular value decomposition of the data matrix $X$ in step 1, an eigenvalue decomposition of the covariance matrix $XX^{T}$ could be computed, which is more efficient if $m\gg d$, but less accurate. The vectors $\xi_{\ell}$ computed by AMUSE are solutions of the eigenvalue problem \eqref{eq:TICA eigenvectors}, since
\begin{align*}
M_{\tsub{TICA}}\,\xi_{\ell} & =\left(XX^{T}\right)^{+}XY^{T}U\Sigma^{-1}w_{\ell}\\
 & =U\Sigma^{-1}\tilde{X}\tilde{Y}^{T}w_{\ell}\\
 & =\lambda_{\ell}\,U\Sigma^{-1}w_{\ell}\\
 & =\lambda_{\ell}\,\xi_{\ell}.
\end{align*}
In the second line, we used the fact that $(XX^{T})^{+}=U\Sigma^{-2}U^{T}$ and in the third that $w_{\ell}$ is an eigenvector of $\overline{M}_{\tsub{TICA}}=\tilde{X}\tilde{Y}^{T}$.

\subsection{Dynamic mode decomposition}

\emph{Dynamic mode decomposition }(DMD) was developed by the fluid dynamics community as a tool to identify coherent structures in fluid flows~\cite{Schmid10}. Since its introduction, several variants and extensions have been proposed, see~\cite{CTR12,TRLBK14,JSN14,BPTK15,KGPS16}. A review of the applications of Koopman operator theory in fluid mechanics can be found in \cite{Mezic13}. DMD can be viewed as a combination of a PCA in the spatial domain and a Fourier analysis in the frequency domain~\cite{BJOK16}. It can be shown that the DMD modes are the Koopman modes for the set of basis functions defined by $\psi(x)=x$. Given again data $X$ and $Y$ as above, the idea behind DMD is to assume that there exists a linear operator $M_{\tsub{DMD}}$ such that $y_{i}=M_{\tsub{DMD}}\,x_{i}$. Since the underlying dynamical system is in general nonlinear, this equation cannot be fulfilled exactly and we want to compute the matrix $M_{\tsub{DMD}}$ in such a way that the Frobenius norm of the deviation is minimized, i.e.,
\begin{equation} \label{eq:Minimization problem DMD}
    \min\norm{Y-M_{\tsub{DMD}}X}_{F}.
\end{equation}
The solution of this minimization problem is given by
\begin{equation} \label{eq:DMD matrix}
    M_{\tsub{DMD}}=YX^{+}=\big(YX^{T}\big)\big(XX^{T}\big)^{+}=C_{\tau}^{T}C_{0}^{+}=M_{\tsub{TICA}}^{T}.
\end{equation}
The eigenvalues and eigenvectors of $M_{\tsub{DMD}}$ are called DMD eigenvalues and modes, respectively. That is, we are solving 
\begin{equation*}
    M_{\tsub{DMD}}\,\xi_{\ell}=\lambda_{\ell}\,\xi_{\ell}.
\end{equation*}
The above equations already illustrate the close relationship with TICA, cf.~\eqref{eq:TICA eigenvectors}. The DMD modes are the right eigenvectors of $M_{\tsub{DMD}}$, whereas the TICA coordinates are defined to be the right eigenvectors of the transposed matrix $M_{\tsub{TICA}}$. Hence, the TICA coordinates are the left eigenvectors of the DMD matrix and the DMD modes the left eigenvectors of the TICA matrix. This is consistent with the results that will be presented in the VAC and EDMD subsections below: The TICA coordinates represent the \emph{Koopman eigenfunctions} projected onto the space spanned by linear basis functions, i.e., $\psi(x)=x$, while the DMD modes are the corresponding \emph{Koopman modes}.

Similar to AMUSE, the DMD modes and eigenvalues can be obtained without explicitly computing $M_{\tsub{DMD}}$ by using a reduced singular value decomposition of $X$. Standard and exact DMD are presented in Algorithm \ref{alg:DMD}. The standard DMD modes are simply the exact DMD modes projected onto the range of the matrix $X$, see~\cite{TRLBK14}.

\begin{algorithm}
\begin{enumerate}
\item \setlength{\itemsep}{0mm}Compute compact SVD of $X$, given by $X=U\,\Sigma\,V^{T}$.
\item Define $\overline{M}_{\tsub{DMD}}=U^{T}YV\Sigma^{-1}$.
\item Compute eigenvalues and eigenvectors of $\overline{M}_{\tsub{DMD}}$, i.e., $\overline{M}_{\tsub{DMD}}\,w_{\ell}=\lambda_{\ell}\,w_{\ell}$.
\item The DMD mode corresponding to the eigenvalue $\lambda_{\ell}$ is defined as
\begin{alignat*}{2}
\text{a) } & \xi_{\ell}=Uw_{\ell}. &  & \qquad\text{(Standard DMD)}\\
\text{b) } & \xi_{\ell}=\frac{1}{\lambda}YV\Sigma^{-1}w_{\ell}. &  & \qquad\text{(Exact DMD)}
\end{alignat*}
\end{enumerate}
\vspace*{-0.75\baselineskip}\caption{Standard and exact DMD.}
\label{alg:DMD}
\end{algorithm}

\begin{rem}
TICA and standard DMD are closely related. When comparing with the AMUSE formulation, we obtain
\begin{equation*}
    \overline{M}_{\tsub{TICA}}=\tilde{X}\tilde{Y}^{T}=\Sigma^{-1}U^{T}XY^{T}U\Sigma^{-1}=\Sigma\,U{}^{T}M_{\tsub{TICA}}\,U\Sigma^{-1}=:W\Lambda W^{-1}
\end{equation*}
and
\begin{equation*}
    \overline{M}_{\tsub{DMD}}=U^{T}YV\Sigma^{-1}=U^{T}M_{\tsub{DMD}}\,U=U^{T}M_{\tsub{TICA}}^{T}U=:\tilde{W}\tilde{\Lambda}\tilde{W}^{-1}.
\end{equation*}
The TICA coordinates are given by $\Xi=U\Sigma^{-1}W$ and the standard DMD modes by $\tilde{\Xi}=U\tilde{W}$ so that \textendash{} except for the scaling $\Sigma^{-1}$ \textendash{} AMUSE and standard DMD use the same projection, the main difference is that the former computes the eigenvectors of $M_{\tsub{TICA}}$ and the latter the eigenvectors of the transposed matrix $M_{\tsub{TICA}}^{T}$. As a result, AMUSE could be rewritten to compute the DMD modes if we define $\overline{M}'_{\tsub{DMD}} = \tilde{Y}\tilde{X}^{T} = \Sigma^{-1}U^{T}YX^{T}U\Sigma^{-1}$ in step 3 of the algorithm and $\xi_{\ell}=U\Sigma\,w_{\ell}$ in step 5, where $w_{\ell}$ now denotes the eigenvectors of $\overline{M}'_{\tsub{DMD}}$.
\end{rem}

\subsection{Variational approach of conformation dynamics}

The \emph{variational approach of conformation dynamics} (VAC) \cite{NoNu13,NKPMN14,NSVN15} is a generalization of the frequently used Markov state modeling framework that allows arbitrary basis functions and is similar to the variational approach in quantum mechanics \cite{NKPMN14}. As described above, VAC and EDMD (see below) require \textendash{} in addition to the data \textendash{} a set of basis functions (also called \emph{dictionary}),
given by $\psi$. The variational approach is defined only for reversible systems \textendash{} EDMD does not require this restriction \textendash{} and computes eigenfunctions of $\mathcal{T}_{\tau}$ or $\mathcal{K}_{\tau}$,
respectively. Using the data matrices $\Psi_{X}$ and $\Psi_{Y}$ defined in~\eqref{eq:PsiX PsiY}, $C_{0}$ and $C_{\tau}$ defined in \eqref{eq:C_O and C_tau TICA} for the transformed data can be estimated as
\begin{equation*}
\begin{split}
    C_{0} & =\tfrac{1}{m-1}\sum_{k=1}^{m}\psi(x_{k})\,\psi(x_{k})^{T}=\tfrac{1}{m-1}\Psi_{X}\Psi_{X}^{T},\\
    C_{\tau} & =\tfrac{1}{m-1}\sum_{k=1}^{m}\psi(x_{k})\,\psi(y_{k})^{T}=\tfrac{1}{m-1}\Psi_{X}\Psi_{Y}^{T}.
\end{split}
\end{equation*}
In what follows, let $M_{\tsub{VAC}}=C_{0}^{+}C_{\tau}$ for the transformed data matrices $\Psi_{X}$ and $\Psi_{Y}$. The matrix $M_{\tsub{VAC}}$ can be regarded as a finite-dimensional approximation of $\mathcal{K}_{\tau}$
(or $\mathcal{T}_{\tau}$, since the system is assumed to be reversible; the derivation is shown in Section~\ref{sec:Derivation}), respectively. Eigenfunctions of the operator can then be approximated by the eigenvectors
of the matrix $M_{\tsub{VAC}}$. Let $\xi_{\ell}$ be an eigenvector of $M_{\tsub{VAC}}$, i.e., 
\begin{equation*}
    M_{\tsub{VAC}}\,\xi_{\ell}=\lambda_{\ell}\,\xi_{\ell},
\end{equation*}
and $\varphi_{\ell}(x)=\xi_{\ell}^{*}\psi(x)$, where $^{*}$ denotes the conjugate transpose. Since 
\begin{equation*}
    \mathcal{K}_{\tau}\varphi_{\ell}(x)\approx(M_{\tsub{VAC}}\,\xi_{\ell})^{*}\,\psi(x)=\lambda_{\ell}\,\xi_{\ell}^{*}\,\psi(x)=\lambda_{\ell}\,\varphi_{\ell}(x),
\end{equation*}
we obtain an approximation of the eigenfunctions of $\mathcal{K}_{\tau}$. The derivation will be described in detail in Section~\ref{sec:Derivation}.

\subsection{Extended dynamic mode decomposition}

\emph{Extended dynamic mode decomposition} (EDMD), a generalization of DMD, can be used to compute finite-dimensional approximations of the Koopman operator, its eigenvalues, eigenfunctions, and eigenmodes \cite{WKR15,WRK15}. It was shown in \cite{KKS16} that EDMD can be extended to approximate also eigenfunction of the Perron\textendash Frobenius operator with respect to the density underlying the data points. With the notation introduced above, the minimization problem \eqref{eq:Minimization problem DMD} for the transformed data matrices $\Psi_{X}$ and $\Psi_{Y}$ can be written as
\begin{equation}
    \min\norm{\Psi_{Y}-M_{\tsub{EDMD}}\Psi_{X}}_{F}.\label{eq:Minimization problem EDMD}
\end{equation}
The solution \textendash{} see also \eqref{eq:DMD matrix} \textendash{}
is given by
\begin{equation*}
    M_{\tsub{EDMD}}=\Psi_{Y}\Psi_{X}^{+}=\big(\Psi_{Y}\Psi_{X}^{T}\big)\big(\Psi_{X}\Psi_{X}^{T}\big)^{+}=C_{\tau}^{T}\,C_{0}^{+}=M_{\tsub{VAC}}^{T}.
\end{equation*}
That is, instead of assuming a linear relationship between the data matrices $X$ and $Y$, EDMD aims at finding a linear relationship between the transformed data matrices $\Psi_{X}$ and $\Psi_{Y}$. Eigenfunctions of the Koopman operator are then given by the left eigenvectors $\xi_{\ell}$ of $M_{\tsub{EDMD}}$, i.e.,
\begin{equation*}
    \varphi_{\ell}(x)=\xi_{\ell}^{*}\,\psi(x).
\end{equation*}
The derivation of EDMD can be found in Section~\ref{sec:Derivation}. Since the left eigenvectors of $M_{\tsub{EDMD}}$ are the right eigenvectors of $M_{\tsub{VAC}}$, VAC and EDMD are equivalent as they compute exactly the same eigenvalue and eigenfunction approximations for a data and basis set. 

As shown in \cite{KKS16}, EDMD can also be used to approximate the Perron\textendash Frobenius operator as follows: 
\begin{equation*}
    \tilde{M}_{\tsub{EDMD}} = \big(\Psi_{X}\Psi_{Y}^{T}\big)\big(\Psi_{X}\Psi_{X}^{T}\big)^{+}=C_{\tau}\,C_{0}^{+}.
\end{equation*}
It is important to note that the Perron\textendash Frobenius operator is computed with respect to the density underlying the data matrices. That is, if $X$ is sampled from a uniform distribution, we obtain the eigenfunctions of the Perron\textendash Frobenius operator $\mathcal{P}_{\tau}$. If we, on the other hand, use one long trajectory, the underlying density converges to the equilibrium density $\pi$ and we obtain the eigenfunctions of the Perron\textendash Frobenius operator with respect to the equilibrium density, denoted by $\mathcal{T}_{\tau}$. An approach to compute the equilibrium density from off-equilibrium data is proposed in \cite{WNPKKN16}.

\begin{example} \label{ex:OU EDMD}
Let us consider the Ornstein\textendash Uhlenbeck process introduced in Example~\ref{ex:OU process}. Here, $\alpha=4$ and $D=0.25$. The lag time is defined to be $\tau=1$. We generated $10^{5}$ uniformly distributed test points in $[-2,\,2]$ and used a basis comprising monomials of order up to $10$. With the aid of EDMD, we computed the dominant eigenfunctions of the Perron\textendash Frobenius operator $\mathcal{P}_{\tau}$ and the Koopman operator $\mathcal{K}_{\tau}$ (which is identical to $\mathcal{T}_{\tau}$ here due to reversibility). The results are shown in Figure~\ref{fig:Eigenfunctions OU}. The corresponding eigenvalues are given by
\begin{equation*}
    \lambda_{1}(\tau)=1.00,\quad\lambda_{2}(\tau)=0.37,\quad\lambda_{3}(\tau)=0.13,\quad\lambda_{4}(\tau)=0.049,
\end{equation*}
which is a good approximation of the analytically computed eigenvalues (Example~\ref{ex:OU eigenvalues}).\exampleSymbol
\end{example}

In order to approximate the Koopman modes, let $\varphi(x)=\left[\varphi_{1}(x),\,\dots,\,\varphi_{k}(x)\right]^{T}$ be the vector of eigenfunctions and
\begin{equation*}
    \Xi = \begin{bmatrix}\xi_{1} & \xi_{2} & \dots & \xi_{k}\end{bmatrix}
\end{equation*}
the matrix that contains all left eigenvectors of $M_{\tsub{EDMD}}$. Furthermore, define $B\in\mathbb{R}^{d\times k}$ such that $g(x)=B\,\psi(x)$. That is, the full-state observable is written in terms of the basis functions\footnote{The easiest way to accomplish this is by adding the observables $x_{i}$, $i=1,\dots,d$, to the set of basis functions.}. Since $\varphi(x)=\Xi^{*}\,\psi(x)$, this leads to $g(x)=B\,\psi(x)=B\,(\Xi^{*})^{-1}\varphi(x)$. Thus, the $\ell$th column vector of the matrix $\boldsymbol{\eta} = B\,(\Xi^{*})^{-1}$ represents the Koopman mode $\eta_{\ell}$ required for the reconstruction of the dynamical system, see \eqref{eq:Koopman reconstruction}.

\subsection{Relationships with other methods\label{subsec:Relationships with other methods}}

For particular choices of basis functions, VAC and EDMD are equivalent to other methods (see also~\cite{NSVN15,KKS16}):
\begin{enumerate}
\item If we choose $\psi(x)=x$, we obtain TICA and DMD, respectively. That is, the TICA coordinates are the eigenfunctions of the Koopman operator projected onto the space spanned by linear basis functions and the
DMD modes are the corresponding Koopman modes. (Note that in this case \textbf{$B=I$} and the matrix $\boldsymbol{\eta} = (\Xi^{*})^{-1}$ contains the right eigenvectors of $M_{\tsub{EDMD}}$.) In many applications
of TICA, the basis functions are modified to have zero mean. For reversible processes, this eliminates the stationary eigenvalue $\lambda_{1}=1$ and its eigenfunction $\varphi_{1}\equiv1$. The largest eigenpair
then approximates the slowest dynamical eigenvalue and eigenfunction, respectively. 
\item If the set of basis functions comprises indicator functions $\mathds{1}_{A_{1}},\,\dots,\,\mathds{1}_{A_{k}}$ for a given decomposition of the state space into disjoint sets $A_{1},\dots,A_{k}$,
VAC and EDMD result in Ulam's method \cite{Ulam60} and thus a Markov state model (MSM).
\end{enumerate}
These relationships are shown in Figure~\ref{fig:Relationships}. Detailed examples illustrating the use of VAC and EDMD can be found in \cite{NKPMN14,WKR15,KKS16}.

\begin{figure}
    \centering
    \includegraphics[width=0.75\linewidth]{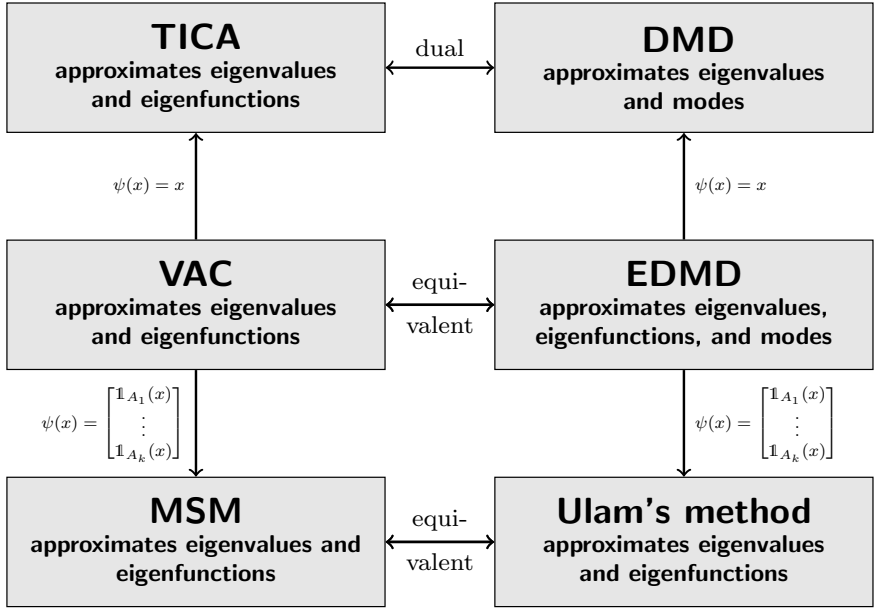}
    \caption{Relationships between data-driven methods. While VAC was derived for reversible dynamical systems, the derivation of EDMD covers non-reversible dynamics as well.}
    \label{fig:Relationships}
\end{figure}

\subsection{Examples}

\subsubsection{Double gyre}

Let us consider the autonomous double gyre, which was introduced in \cite{Shadden05}, given by the SDE
\begin{equation*}
\begin{split}
    d\mathbf{X}_{t} & =-\pi\,A\,\sin(\pi\,\mathbf{X}_{t})\,\cos(\pi\,\mathbf{Y}_{t}) + \varepsilon\,d\mathbf{W}_{t,1},\\
    d\mathbf{Y}_{t} & =\phantom{-}\pi\,A\,\cos(\pi\,\mathbf{X}_{t})\,\sin(\pi\,\mathbf{Y}_{t}) + \varepsilon\,d\mathbf{W}_{t,2}
\end{split}
\end{equation*}
on the domain $\mathbb{X}=[0,2]\times[0,1]$ with reflecting boundary. For $\varepsilon=0$, there is no transport between the left half and the right half of the domain and both subdomains are invariant sets with measure $\frac{1}{2}$, cf.~\cite{FP09,FP14}. For $\varepsilon>0$, there is a small amount of transport due to diffusion and the subdomains are almost invariant. For the Koopman operator $\mathcal{\mathcal{K}_{\tau}}$, this means that for $\varepsilon=0$ the characteristic functions $\tilde{\varphi}_{1}=\mathds{1}_{[0,1]\times[0,1]}$ and $\tilde{\varphi}_{2}=\mathds{1}_{[1,2]\times[0,1]}$ are both eigenfunctions with corresponding eigenvalue $1$. If, on
the other hand, $\varepsilon>0$, then the two-dimensional eigenspace subdivides into two one-dimensional eigenspaces with eigenvalues $\lambda_{1}=1$ and $\lambda_{2}=1-\mathcal{O}(\varepsilon)$ and eigenfunctions $\varphi_{1}=\mathds{1}_{[0,2]\times[0,1]}$ and $\varphi_{2}\approx\tilde{\varphi}_{1}-\tilde{\varphi}_{2}$. Typical trajectories of the system are shown in Figure~\ref{fig:Double gyre}a. Using the parameters $A=0.25$ and $\varepsilon=0.05$, we integrated $10^{5}$ randomly generated test points using the Euler\textendash Maruyama scheme with step size $h=10^{-3}$. 

For the computation of the eigenfunctions, we choose a set of radial basis functions whose centers were given by the midpoints of an equidistant $50\times25$ box discretization, and a lag time $\tau=3$. The resulting nontrivial leading eigenfunctions of the Koopman operator computed with EDMD are shown in Figure~\ref{fig:Double gyre}b. The two almost invariant sets are clearly visible. The eigenfunctions of the Perron\textendash Frobenius operator exhibit similar patterns (but ``rotating'' in the opposite direction).

\begin{figure}[tbph]
    \centering
    \subfiguretitle{a)}
    \includegraphics[width=0.6\textwidth]{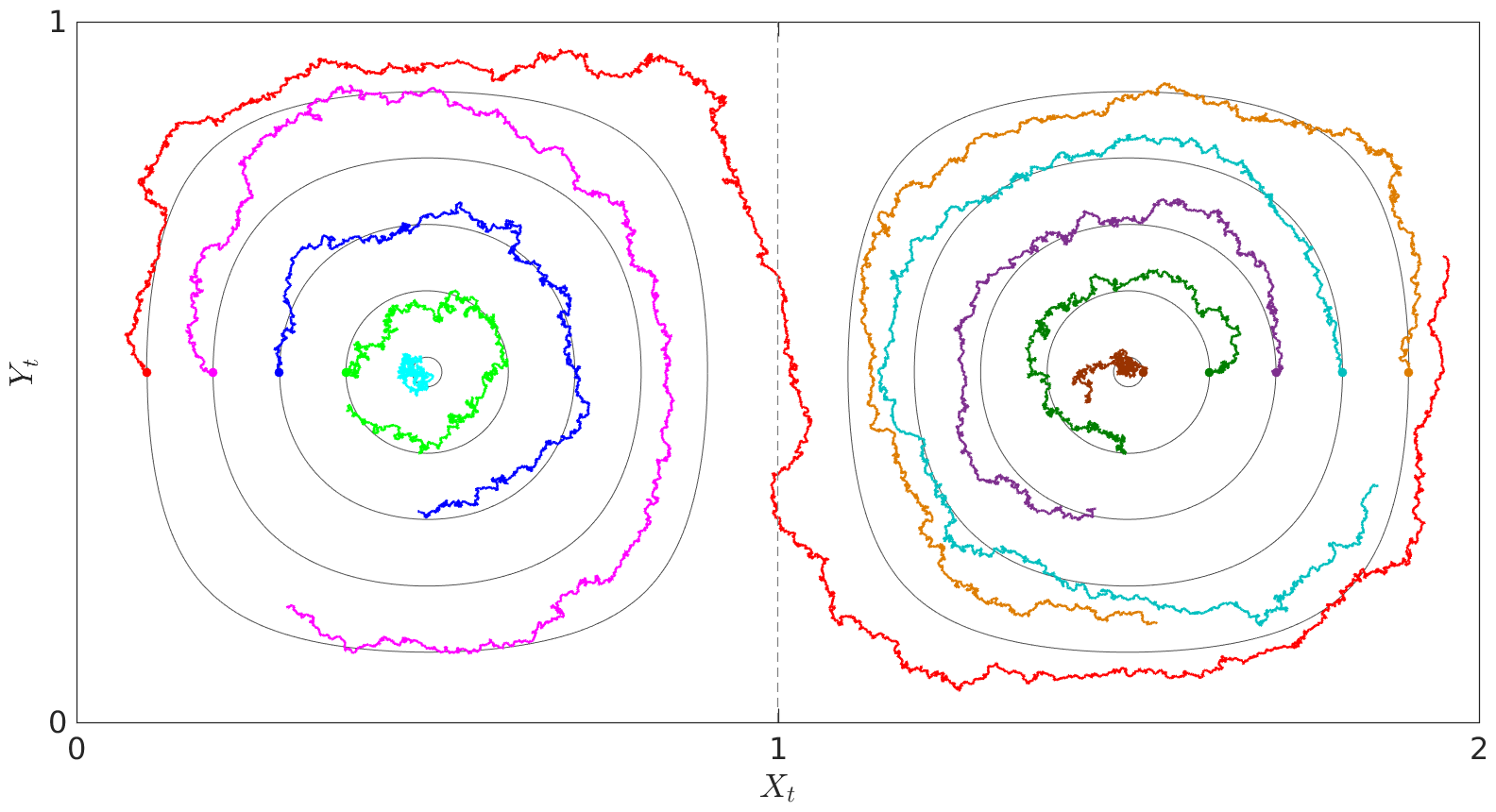} \\
    \subfiguretitle{b)}
    \includegraphics[width=1\textwidth]{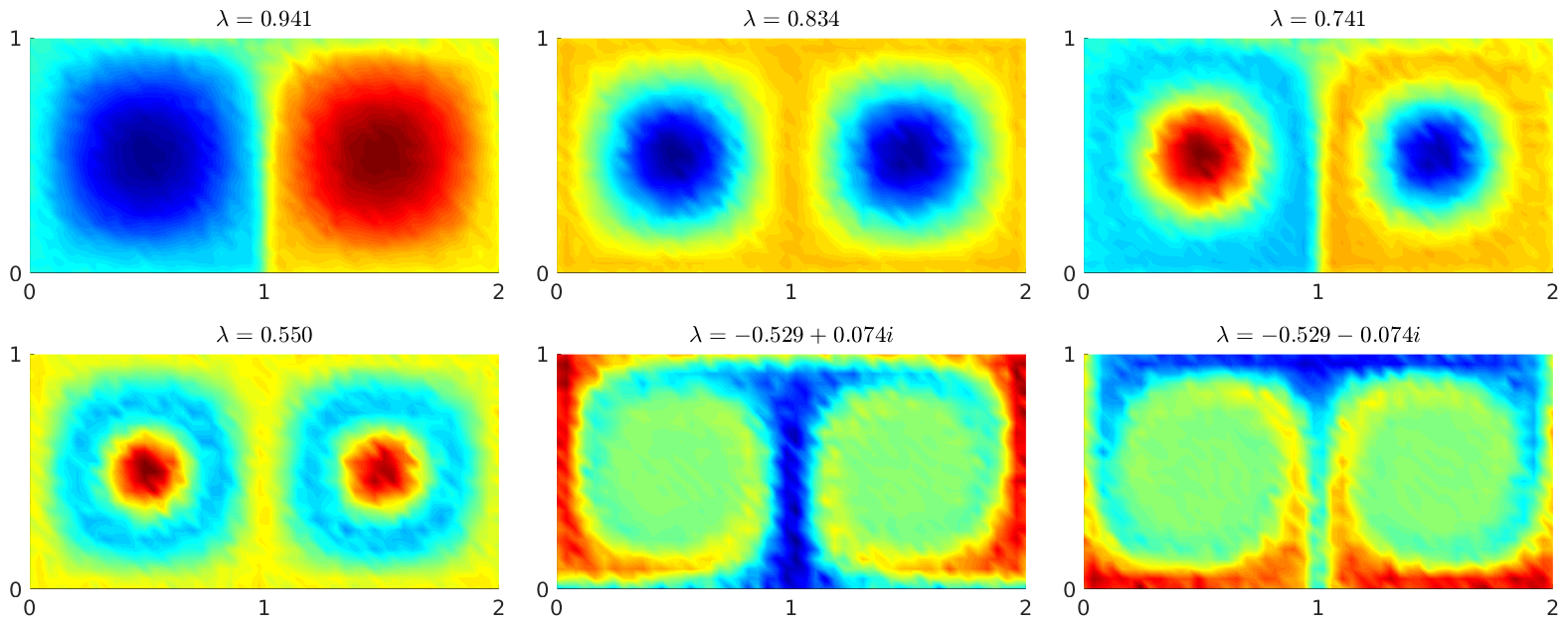}
    \vspace{1ex}
    \includegraphics[width=0.5\textwidth]{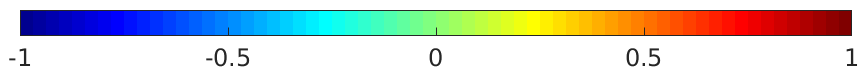}
    \caption{a) Typical trajectories (of different lengths) of the double gyre system for $\varepsilon=0.05$. The initial states are marked by dots. Due to the diffusion term, particles can cross the separatrix (dashed line). The gray lines show the trajectories with the same initial conditions for $\varepsilon=0$. b) Leading eigenfunctions of the Koopman operator associated with the double gyre system computed using EDMD.}
    \label{fig:Double gyre}
\end{figure}

\subsubsection{Deca alanine}

As a second example, we illustrate what has become a typical workflow for the application of VAC/EDMD in molecular dynamics, using deca alanine as a model system. Deca alanine is a small peptide comprised
of ten alanine residues, it has been used as a test system many times before. Here, we analyze equilibrium simulations of $3\,\mathrm{\mu s}$ total simulation time using the Amber03 force field, see \cite{NKPMN14,NSVN15} for the detailed simulation setup. A set of important quantities for our analysis are the leading \textit{implied timescales}
\begin{equation}
    t_m = -\frac{\tau}{\log|\lambda_{m}|},\label{eq:implied_timescale}
\end{equation}
for $ m=2,3,\ldots $. Implied timescales are independent of the lag time \cite[Theorem 2.2.4]{Pazy83}. However, if they are estimated using \eqref{eq:implied_timescale} and an approximation to the eigenvalues $\lambda_{m}$ obtained from VAC/EDMD, the timescales will be underestimated (see Section~\ref{subsec:variational_principle}) and the error will decrease as a function of the lag time \cite{DSS12}. Approximate convergence of implied timescales with increasing lag time has become a standard model validation criterion in molecular dynamics \cite{PWSKSHSCSF11}.

In the first step, a set of internal molecular coordinates is extracted from the simulation data, to which TICA is applied. In our example, we select all 16 backbone dihedral angles as internal molecular coordinates. Figure~\ref{fig:deca_alanine}a shows the first five implied timescales estimated by TICA as a function of the lag time $\tau$.

Next, a first dimension reduction is performed, where the data is projected onto the leading $M$ TICA eigenvectors. The number $M$ is selected by the criterion of total kinetic variance, that is, $M$ is the smallest number such that the cumulative sum of the first $M$ squared eigenvalues exceeds 95 per cent of the total sum of squared eigenvalues \cite{NC15}. Figure~\ref{fig:deca_alanine}c shows the resulting dimension $M$ as a function of the lag time.

As a third step, the reduced data set is discretized by application of a clustering method. In our case, we use $k$-means clustering to assign the data to 50 discrete states. A Markov state model (MSM, equivalent to Ulam's method, see above) is estimated from the discretized time series. We show the first five implied timescales from the MSM in Figure~\ref{fig:deca_alanine}c and observe that estimates improve compared to the TICA approximations. Also, timescale estimates converge for lag times $\tau\geq4\,\mathrm{ns}$.

Finally, we use the converged model at lag time $\tau=4\,\mathrm{ns}$ for further analysis. As the slowest implied timescale $t_{2}$ dominates all others, and as it is the only one which is larger than the lag time used for analysis (indicated by the gray line in Figure~\ref{fig:deca_alanine}c), we attempt to extract a two-state model that captures the essential dynamics. We employ the PCCA+ algorithm \cite{DW05,RW13} to coarse grain all MSM states into two macrostates. Inspection of randomly selected trajectory frames belonging to each macrostate reveals that the slow dynamical process in the data corresponds to the formation of a helix, see FigureR~\ref{fig:deca_alanine}d. It should be noted that this coarse graining works well for visualization purposes, but some details need to be taken into account. In fact, PCCA performs a fuzzy assignment of MSM states to macrostates, where each MSM state
belongs to each macrostate with a certain \textit{membership} in $[0,1]$. We simply assign each MSM state to the macrostate with maximal membership here. Alternatively, we could also use a \textit{hidden Markov model} (HMM) to perform the coarse graining \cite{NWPP13}.

\begin{figure}[htbp]
    \begin{minipage}{0.49\textwidth}
        \centering
        \subfiguretitle{a)}
        \includegraphics[width=\textwidth]{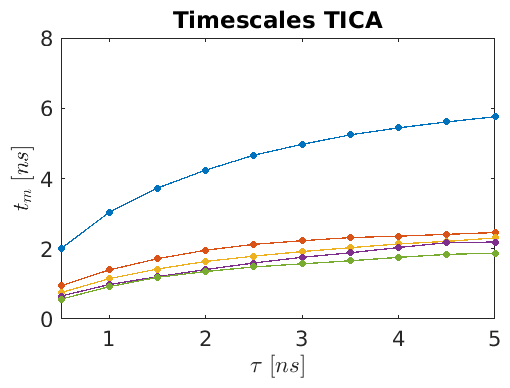}
    \end{minipage}
    \begin{minipage}{0.49\textwidth}
        \centering
        \subfiguretitle{c)}
        \includegraphics[width=\textwidth]{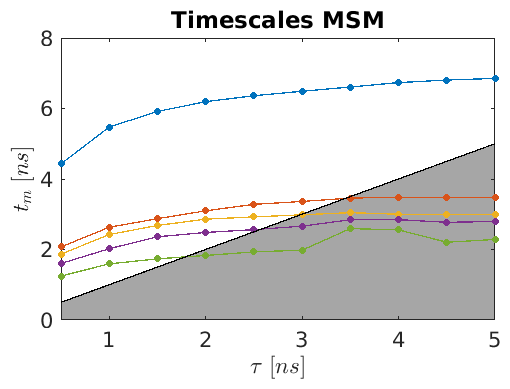}
    \end{minipage}
    \\[1ex]
    \begin{minipage}[t]{0.49\textwidth}
        \centering
        \subfiguretitle{b)}
        \includegraphics[width=\textwidth]{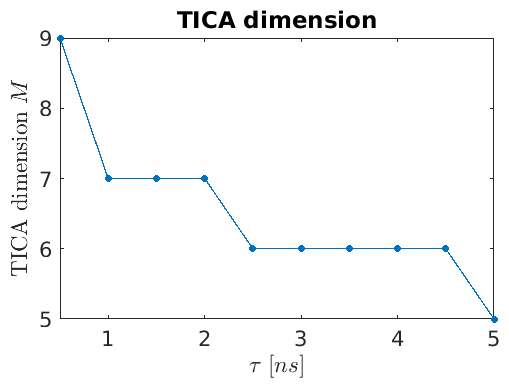}
    \end{minipage}
    \begin{minipage}[t]{0.49\textwidth}
        \centering
        \subfiguretitle{d)}
        \vspace*{4ex}
        \includegraphics[width=0.85\textwidth]{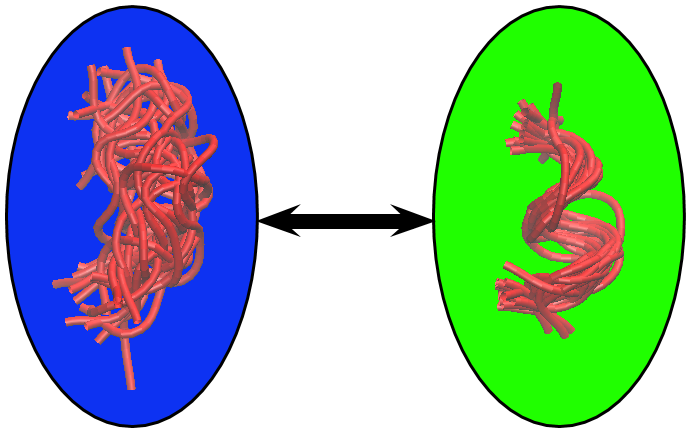}
    \end{minipage}
    \caption{Illustration of standard EDMD workflow in molecular dynamics using the deca alanine model system. a) Leading implied timescales $t_m$ (in nanoseconds) as estimated by TICA as a function of the lag time.b) Effective dimension $M$ selected by applying the criterion of total kinetic variance to the TICA eigenvalues. c) Leading implied timescales $t_{m}$ estimated by a Markov state model after projecting the data onto the first $M$ TICA eigenvectors and discretizing this data set into 50 states using $k$-means. d) Simple visualization of effective coarse grained dynamics. All MSM states are assigned to two macrostates using the PCCA algorithm. An overlay of representative structures from both macrostates shows that the dynamics between them corresponds to helix formation. Macrostates are drawn proportionally to their stationary probability. }
\label{fig:deca_alanine}
\end{figure}

\section{Derivations\label{sec:Derivation}}

In this section, we will show how VAC and EDMD as well as their respective special cases TICA and DMD can be derived and how these methods are related to eigenfunctions and eigenmodes of transfer operators.

\subsection{General dynamical systems}

Let us begin with general, not necessarily reversible dynamical systems. In order to be able to compute eigenfunctions of transfer operators numerically, the infinite-dimensional operators are projected onto a finite-dimensional space. We will briefly outline how the EDMD minimization problem \eqref{eq:Minimization problem EDMD} leads to an approximation of the Koopman operator.

\begin{thm}
\label{thm:EDMD minimization problem}Let the process $\{\mathbf{X}_{t}\}_{t\ge0}$ be Feller-continuous\footnote{A process $\{\mathbf{X}_{t}\}_{t\ge0}$ is called Feller-continuous if the mapping $x\mapsto\mathbb{E}[g(\mathbf{X}_{t})\vert\mathbf{X}_{0}=x]$ is continuous for any fixed continuous function $g$. This implies, that the Koopman operator of a Feller-continuous process has a well defined restriction from $L^{\infty}(\mathbb{X})$ to the set of continuous functions. Any stochastic process generated by an Itô stochastic differential equation with Lipschitz-continuous coefficients is Feller-continuous \cite[Lemma 8.1.4]{Oks03}.}. Let $\psi_{i}$, $i=1,\,\dots,\,k$, be the set of at least piecewise continuous basis functions of the finite-dimensional linear space $\mathbb{V}$. Let the empirical distribution of the data points $x_{1},x_{2},\ldots$ converge weakly to the density $\rho$. Then the minimization problem
\begin{equation*}
    \min_{K\in\mathbb{R}^{k\times k}}\frac{1}{m}\sum_{j=1}^{m}\norm{\psi(y_{j})-K^{T}\psi(x_{j})}_{2}^{2}
\end{equation*}
converges, as $m\to\infty$, almost surely to $\min_{\hat{K}}\sum_{i=1}^{k}\|\mathcal{\mathcal{K_{\tau}}}\psi_{i}-\hat{K}\psi_{i}\|_{\rho}^{2}$, where the minimization is over all linear mappings $\hat{K}:\mathbb{V}\to\mathbb{V}$.
\end{thm}

\begin{proof}
Let $f=\sum_{i=1}^{k}a_{i}\,\psi_{i}=a^{T}\psi\in\mathbb{V}$ be an arbitrary function, where $a=[a_{1},\,\dots,\,a_{k}]^{T}$. For a single data point $x_{j}$, we have for a linear mapping $\hat{K}:\mathbb{V}\to\mathbb{V}$
with matrix representation $K\in\mathbb{R}^{k\times k}$ that 
\begin{equation*}
    \hat{K}f(x_{j})=\sum_{i=1}^{k}(K\,a)_{i}\,\psi_{i}(x_{j})=a^{T}K^{T}\psi(x_{j}).
\end{equation*}
Here, the $i$th column of the matrix $ K $ corresponds to $ \hat{K} \psi_i $. Thus, we obtain
\begin{eqnarray*}
\frac{1}{m}\sum_{j=1}^{m}\norm{\psi(y_{j})-K^{T}\psi(x_{j})}_{2}^{2} & = & \sum_{i=1}^{k}\frac{1}{m}\sum_{j=1}^{m}(\psi_{i}(y_{j})-\hat{K}\psi_{i}(x_{j}))^{2}\\
 & \stackrel{m\to\infty}{\longrightarrow} & \sum_{i=1}^{k}\int_{\mathbb{X}}(\mathbb{E}[\psi_{i}(\mathbf{X}_{\tau})\,\vert\,\mathbf{X}_{0}=x]-\hat{K}\psi_{i}(x))^{2}\rho(x)\,dx\\
 & = & \sum_{i=1}^{k}\|\mathcal{\mathcal{K_{\tau}}}\psi_{i}-\hat{K}\psi_{i}\|_{\rho}^{2},
\end{eqnarray*}
where the convergence for $m\to\infty$ is almost sure. From the first line to the second we used that the $y_{j}$ are realizations of the random variables $\mathbf{X}_{\tau}$ given $\mathbf{X}_{0}=x_{j}$, that $\mathbf{X}_{\tau}$ is a Feller-continuous process, that the $\psi_{i}$ are (piecewise) continuous functions, and that the sampling process of $x_{j}$ is independent of the noise process that decides over $\mathbf{X}_{\tau}$ given $\mathbf{X}_{0}=x_{j}$.
\end{proof}

With the aid of the data matrices $\Psi_{X}$ and $\Psi_{Y}$ defined in \eqref{eq:PsiX PsiY}, this minimization problem can be written as
\begin{equation*}
    \min\norm{\Psi_{Y}-K^{T}\Psi_{X}}_{F}^{2},
\end{equation*}
which is identical to \eqref{eq:Minimization problem EDMD}, where now $K^{T}=M_{\tsub{EDMD}}$. Thus, the transposed EDMD matrix $M_{\tsub{EDMD}}$ is an approximation of the Koopman operator. A similar setup allows for the approximation of the Perron\textendash Frobenius operator with respect to the data point density $\rho$. For details, we refer to \cite[Appendix A]{KKS16}. Note, however, that although the Perron\textendash Frobenius and Koopman operators are adjoint, the matrix representation of the discrete Perron\textendash Frobenius operator will in general not just be the transposed of the matrix $K$, unless the ansatz functions $\psi_{i}$ are orthonormal with respect to $\langle\cdot,\,\cdot\rangle_{\rho}$.

If the dynamical system is deterministic, we can already interpret the minimization \eqref{eq:Minimization problem EDMD} for finite values of $m$. As shown, e.g., in \cite{KKS16,KoMe17}, the solution of \eqref{eq:Minimization problem EDMD} is a Petrov\textendash Galerkin projection of the Koopman operator on the ansatz space $\mathbb{V}$.

\subsection{Reversible dynamical systems}

Let us now assume that the system is reversible. That is, it holds that $\pi(x)\,p_{\tau}(x,y)=\pi(y)\,p_{\tau}(y,x)$ for all $x$ and $y$.

\subsubsection{Variational principle for the Rayleigh trace\label{subsec:variational_principle}}

We can also derive a variational formulation for the first $M$ eigenvalues of the Koopman operator $\mathcal{K}_{\tau}$ in the reversible setting. It is a standard result for self-adjoint operators on a Hilbert space with bounded eigenvalue spectrum, see, e.g., \cite{Ban80}:

\begin{prop} \label{prop:Variational principle Rayleigh trace}
Assume that $1=\lambda_{1}>\lambda_{2}\geq\ldots\geq\lambda_{M}$ are the dominant eigenvalues of the Koopman operator $\mathcal{K}_{\tau}$ on $L_{\pi}^{2}$. Then
\begin{equation} \label{eq:Variational formulation Rayleigh trace}
    \begin{split}
        \sum_{\ell=1}^{M}\lambda_{\ell} & =\sup\sum_{\ell=1}^{M}\langle\mathcal{K}_{\tau}v_{\ell},v_{\ell}\rangle_{\pi},\\
        \langle v_{\ell},v_{\ell^{'}}\rangle_{\pi} & =\delta_{\ell\ell^{'}}
    \end{split}
\end{equation}
The sum of the first $M$ eigenvalues maximizes the Rayleigh trace, which is the sum on the right-hand side of \eqref{eq:Variational formulation Rayleigh trace} over all selections of $M$ orthonormal functions $v_{\ell}$. The maximum is attained for the first $M$ eigenfunctions $\varphi_{1},\ldots,\varphi_{M}$.
\end{prop}

\begin{proof}
The $M$-dimensional space $\mathbb{V}$ spanned by the functions $v_{\ell}$ must contain an element $u_{M}$ which is orthonormal to the first $M-1$ eigenfunctions $\varphi_{\ell}$, i.e., $\langle u_{M},\varphi_{\ell}\rangle_{\pi} = 0$, $\ell=1,\ldots,M-1$, and $ \norm{u_M}_\pi = 1 $. By the standard Rayleigh principle for self-adjoint operators
\begin{equation*}
    \langle K_{\tau}u_{M},u_{M}\rangle_{\pi}\leq\lambda_{M}.
\end{equation*}
Next, determine a normalized element $u_{M-1}$ of the orthogonal complement of $u_{M}$ in $\mathbb{V}$ with $\langle u_{M-1},\varphi_{\ell}\rangle_{\pi}=0$, $\ell=1,\ldots,M-2$. Again, we can invoke the Rayleigh principle to find
\begin{equation*}
    \langle\mathcal{K}_{\tau}u_{M-1},u_{M-1}\rangle_{\pi}\leq\lambda_{M-1}.
\end{equation*}
Repeating this argument another $M-2$ times provides an orthonormal basis $u_{1},\ldots,u_{M}$ of the space $\mathbb{V}$ such that
\begin{equation*}
    \sum_{\ell=1}^{M} \langle\mathcal{K}_{\tau}u_{\ell},u_{\ell}\rangle_{\pi} \leq \sum_{\ell=1}^{M}\lambda_{\ell}.
\end{equation*}
As the Rayleigh trace is independent of the choice of orthonormal basis for the subspace $\mathbb{V}$, and the space itself was arbitrary, this proves \eqref{eq:Variational formulation Rayleigh trace}. Clearly, the maximum is attained for the first $M$ eigenfunctions.
\end{proof}

Proposition \ref{prop:Variational principle Rayleigh trace} motivates the variational approach developed in \cite{NoNu13,NueskeEtAl_JCTC14_Variational} to maximize the Rayleigh trace restricted to some fixed space of ansatz
functions:

\begin{prop}
Let $\mathbb{V}$ be a space of $k$ linearly independent ansatz functions $\psi_{i}$ given by a dictionary as above. The set of $M\leq k$ mutually orthonormal functions $f^{\ell}=\sum_{i=1}^{k}a_{i}^{\ell}\,\psi_{i}$ which maximize the Rayleigh trace of the Koopman operator restricted to $\mathbb{V}$ is given by the first $M$ eigenvectors of the generalized eigenvalue problem
\begin{equation}
    C_{\tau}\,a^{\ell}=\hat{\lambda}_{\ell}\,C_{0}\,a^{\ell},\label{eq:Generalized_EV_Problem}
\end{equation}
where $a^{\ell}=\left(a_{i}^{\ell}\right)_{i=1}^{k}$, and the matrices $C_{\tau},\,C_{0}$ are given by
\begin{equation*}
    \begin{split}
        (C_{\tau})_{ij} & =\langle\mathcal{K}_{\tau}\psi_{i},\,\psi_{j}\rangle_{\pi},\\
        (C_{0})_{ij} & =\langle\psi_{i},\,\psi_{j}\rangle_{\pi}.
    \end{split}
\end{equation*}
\end{prop}

\begin{proof}
First, note that for any functions $f=\sum_{i=1}^{k}a_{i}\,\psi_{i}$ and $g=\sum_{i=1}^{k}b_{i}\,\psi_{i}$, we have that
\begin{equation*}
    \begin{split}
        \langle\mathcal{K}_{\tau}f,g\rangle_{\pi} & =a^{T}C_{\tau}\,b,\\
        \langle f,g\rangle_{\pi} & =a^{T}C_{0}\,b.
    \end{split}
\end{equation*}
Let us assume that the ansatz functions are mutually orthonormal, i.e., $C_{0}=I$. Then, maximization of the Rayleigh trace is equivalent to finding $M$ vectors $a^{\ell}$, such that $\left(a^{\ell}\right)^{T}a^{\ell^{'}}=\delta_{\ell\ell^{'}}$ and
\begin{equation*}
    \sum_{\ell=1}^{M}\left(a^{\ell}\right)^{T}C_{\tau}\,a^{\ell}=\sum_{\ell=1}^{M}\langle C_{\tau}\,a^{\ell},a^{\ell}\rangle
\end{equation*}
is maximal. By Proposition \ref{prop:Variational principle Rayleigh trace} applied to the operator $C_{\tau}$ on $\mathbb{R}^{N}$, the vectors $a^{\ell}$ are given by the first $M$ eigenvectors of $C_{\tau}$. In the general case, transform the basis functions into a set of mutually orthonormal functions $\tilde{\psi}_{i}$ via $\tilde{\psi}_{i}=\sum_{j=1}^{k}C_{0}^{-1/2}(j,i)\,\psi_{j}$. For the transformed basis, we need to compute the eigenvectors $\tilde{a}^{\ell}$ of
\begin{equation*}
    C_{0}^{-1/2}C_{\tau}C_{0}^{-1/2}\,\tilde{a}^{\ell}=\hat{\lambda}_{\ell}\,\tilde{a}^{\ell}.
\end{equation*}
This is equivalent to the generalized eigenvalue problem \eqref{eq:Generalized_EV_Problem}, the relation between the eigenvectors is given by
\begin{equation*}
    a^{\ell} = C_{0}^{-1/2}\,\tilde{a}^{\ell}.\tag*{\qedhere}
\end{equation*}
\end{proof}

\section{Conclusion\label{sec:Conclusion}}

In this review paper, we established connections between different data-driven model reduction and transfer operator approximation methods developed independently by the dynamical systems, fluid dynamics, machine learning, and molecular dynamics communities. Although the derivations of these methods differ, we have shown that the resulting algorithms share many similarities. 

DMD, TICA and MSMs are popular methods to approximate the dynamics of high-dimensional systems. Due to their simple basis functions, they conduct relatively rough approximations, but when only a few spectral components are required, the approximation error can be controlled by choosing sufficiently large lag times $\tau$ \cite{SarichNoeSchuette_MMS09_MSMerror}. The more general methods VAC and EDMD are better suited to obtain accurate approximations of eigenfunctions. However, to ensure such an accurate approximation, one would have to deploy multiple basis functions in all coordinates and their combinations, which is unfeasible for high-dimensional systems, and would also lead to overfitting when estimating the eigenfunctions of the Koopman operator from a finite data set \cite{MP15}.

A natural approach to mitigate these problems is to construct an iterative or ``deep'' approach in which the dynamical systems subspace in which a high resolution of basis functions is required is found by multiple successive analysis steps. A common approach is to first reduce the dimension by an inexpensive method such as TICA, in order to have a relatively low-dimensional space in which the eigenfunctions are approximated with a higher-resolution method. Another possibility is to exploit low-rank tensor approximations of transfer operators and their eigenfunctions. Tensor- and sparse-grid based reformulations of some of the methods described in this paper can be found in~\cite{NSVN15,KS16,KGPS16}, and in~\cite{JuKo09}, respectively. The efficiency of these tensor decomposition approaches depends strongly on the coupling structure; strong coupling between different variables typically leads to high ranks. Furthermore, some tensor formats also depend on the ordering of variables and a permutation of the variable's indices would lead to different tensor decompositions. Yet another approach might be to exploit sparsity-promoting methods using $L_{1}$-regularization techniques. Basis functions that are not required to represent the eigenfunctions of an operator can thus be eliminated and refined adaptively. Moreover, dictionary-learning methods could be applied to \emph{learn} a basis set and to adapt the dictionary to specific data \cite{MBPS09}. Future work includes evaluating and combining different dimensionality reduction, tensor decomposition, and sparsification methods to mitigate the curse of dimensionality.

\section*{Acknowledgements}

This research has been partially funded by Deutsche Forschungsgemeinschaft (DFG) through grant CRC 1114 \emph{``Scaling Cascades in Complex Systems''}, Project A04 \emph{``Efficient calculation of slow and stationary scales in molecular dynamics''} and Project B03 \emph{``Multilevel coarse graining of multi-scale problems''}, and by the Einstein Foundation Berlin (Einstein Center ECMath). Furthermore, we would like to thank the reviewers for their helpful comments and suggestions.

\bibliographystyle{unsrt}
\bibliography{TICA_vs_DMD,own}

\end{document}